\documentclass[11pt,a4paper,reqno]{amsart}
\usepackage{CJKutf8}
\usepackage{tabls}
\usepackage{url}
\usepackage[linktocpage = true]{hyperref}
\hypersetup{
 colorlinks = true,
 citecolor = blue,
}
\usepackage{color}
\definecolor{red}{rgb}{1,0,0}
\definecolor{green}{rgb}{0,1,0}
\definecolor{blue}{rgb}{0,0,1}
\definecolor{refkey}{gray}{.625}
\definecolor{labelkey}{gray}{.625}
\usepackage[lite]{amsrefs}
\usepackage[a4paper]{geometry}
\usepackage{amsfonts}
\usepackage{amssymb}
\usepackage{mathrsfs}
\usepackage{amsmath}
\usepackage{amsthm}
\usepackage{amscd}
\usepackage{enumerate}
\usepackage{paralist}
\usepackage{graphicx}
\usepackage{mathtools}
\usepackage{kantlipsum}
\usepackage{tikz-cd}
\usepackage{tabularx}
\allowdisplaybreaks

\newcommand{\id}{\mathrm{Id}}
\newcommand{\R}{\mathbb{R}}

\newcommand{\g}{\mathfrak{g}}
\newcommand{\h}{\mathfrak{h}}

\newcommand{\saut}{\operatorname{sAut}}
\newcommand{\K}{\mathbb{K}}

\newcommand{\Mf}{\operatorname{Mfld}}

\makeatletter
 \def\title@font{\normalsize\bfseries}
 \let\ltx@maketitle\@maketitle
 \def\@maketitle{\bgroup%
 \let\ltx@title\@title%
 \def\@\title{\resizebox{\textwidth}{!}{%
  \mbox{\title@font\ltx@title}%
 }}%
 \ltx@maketitle%
 \egroup}
\makeatother

\theoremstyle{plain}
\newtheorem{lem}[equation]{Lemma}
\newtheorem{Cor}[equation]{Corollary}
\newtheorem{Thm}[equation]{Theorem}
\newtheorem{theorem}{Theorem}
\newtheorem{Def}[equation]{Definition}
\newtheorem{def-prop}[equation]{Definition-Proposition}
\newtheorem{prop}[equation]{Proposition}
\newtheorem{prop-def}[equation]{Proposition-Definition}
\newtheorem{Ex}[equation]{Example}
\newtheorem{Rem}[equation]{Remark}

\numberwithin{equation}{section}
\begin{document}
\def\a{\mathfrak{a}}
\def\C{\mathbb{C}}
\def\CE{\mathrm{CE}}
\def\D{\mathcal{D}}
\def\E{\mathscr{E}}
\def\ev{\mathrm{ev}}
\def\F{\mathscr{F}}
\def\g{\mathfrak{g}}
\def\H{\textbf{H}}
\def\j{\mathscr{J}}
\def\L{\mathcal{L}}
\def\M{\mathcal{M}}
\def\m{\mathfrak{m}}
\def\t{\mathfrak{t}}
\def\O{\mathcal{O}}
\def\P{\mathcal{P}}
\def\r{\mathcal{R}}
\def\U{\mathcal{U}}
\def\v{\mathscr{A}}
\def\w{\mathfrak{W}}
\def\X{\mathbb{X}}
\def\Y{\mathbb{Y}}
\def\spec{\text{spec}}
\def\Im{\text{Im}}
\def\coker{\operatorname{coker}}
\def\Ext{\operatorname{Ext}}
\def\End{\operatorname{End}}
\def\id{\mathrm{Id}}
\def\Der{\operatorname{Der}}
\def\IDer{\operatorname{IDer}}
\def\Hom{\operatorname{Hom}}
\def\pr{\operatorname{pr}}
\def\Map{\operatorname{Map}}
\def\Mod{\operatorname{Mod}}
\def\sgn{\operatorname{sgn}}
\def\sh{\operatorname{sh}}

\newcommand{\Linfty}{L_\infty}
\newcommand{\piepartial}{\eth}
\newcommand{\inputvariable}{\cdot}

\newcommand{\Artinring}{\mathfrak{a}}
\newcommand{\MC}{\operatorname{MC}}
\newcommand{\MCgA}{\textrm{MC}_\g(\v)}
\newcommand{\MCLPgA}{\textrm{MC}_{\Omega^\bullet_A(B)}(\v)}
\newcommand{\MCfunctor}{{\rm MC}_{\Omega^\bullet_A(B)}}
\newcommand{\algDefFctLA}{{\rm algDef}_{{(L,A)}}}
\newcommand{\weakDefFctLA}{{\rm wkDef}_{{(L,A)}}}
\newcommand{\firstDef}{{\rm Def}_\g}
\newcommand{\MCfunctorg}{{\rm MC}_\g}
\newcommand{\pairing}[2]{\langle #1,#2\rangle}
\newcommand{\rank}{\mathrm{rank}}
\newcommand{\dA}{d_A}
\newcommand{\Linftythree }{L_{\leqslant 3}}
\newcommand{\infDefFctLA}{{\rm wDef}_{{(L,A)}}}
\newcommand{\sinfDefFctLA}{{\rm sDef}_{{(L,A)}}}
\newcommand{{\newboundary}}{\kappa}
\newcommand{\Diff}{\mathrm{Diff}}
\newcommand{\CinfMK}{C^\infty(M,\K)}
\newcommand{\pra}{\mathrm{Pr}_A}
\newcommand{\prav}{\mathrm{Pr}_A}
\newcommand{\prb}{\mathrm{Pr}_B}
\newcommand{\prbv}{\mathrm{Pr}_B}
\newcommand{\deltap}{\delta}
\newcommand{\deltas}{\sigma(\delta)}
\newcommand{\KScorrespondence }{\mathrm{KSc}}
\newcommand{\KSmap}{\mathrm{KS}}
\newcommand{\ArtCat}{\textbf{Art}}
\newcommand{\PhiA}{\Phi_A}
\newcommand{\PhiB}{\Phi_B}
\newcommand{\basic}{\mathfrak{c}}
\newcommand{\adLb}{{\mathrm{L}_b  }}
\newcommand{\adLa}{{\mathrm{L}_a  }}
\newcommand{\maximalidealofartin}{\m_\v}
 \newcommand{\bigiota}{\mathcal{I}}
\newcommand{\SpecofA}{\mathbb{P}}
\newcommand{\Spec}{\operatorname{Spec}}
\newcommand{\FINISH}{\textcolor{red}{FINISH}}

\newcommand{\TMK}{T_M^{\K}}
\newcommand{\dCE}{d_{\mathrm{CE}}}

\title{Infinitesimal deformations of Lie algebroid pairs}
\author{Dadi Ni}
\address{School of Mathematics and Statistics, Henan University,  China} 
\email{\href{mailto:nidd@henu.edu.cn}{nidd@henu.edu.cn}}

\author{Zhuo Chen}
\address{Department of Mathematics, Tsinghua University,  China}
\email{\href{mailto:chenzhuo@tsinghua.edu.cn}{chenzhuo@tsinghua.edu.cn}}

\author{Chuangqiang Hu}
\address{School of Mathematics, Sun Yat-sen University,  China}
\email{\href{huchq5@mail.sysu.edu.cn}{huchq5@mail.sysu.edu.cn}}

\author{Maosong Xiang}
\address{School of Mathematics and Statistics, Center for Mathematical Sciences, Huazhong University of Science and Technology,  China}
\email{\href{mailto: msxiang@hust.edu.cn}{msxiang@hust.edu.cn}}
\thanks{Research partially supported by National Key R\&D Program of China [2022YFA1006200],  the NSFC grants 12071241, 12441107, and National Natural Science Foundation of Henan Province grant 252300421766.}

\begin{abstract}
We study infinitesimal deformations of Lie algebroid pairs in the category of smooth manifolds enriched with a local Artinian $\K$-algebra.
Given a Lie algebroid pair $(L,A)$, i.e. a Lie algebroid $L$ together with a Lie subalgebroid $A$, we investigate  isomorphism classes of infinitesimal deformations of $(L,A)$ modulo automorphisms from exponentials of derivations of $L$ and those from the exponentials of inner derivations of $L$, respectively.
For the associated two deformation functors, we find the associated governing $L_\infty$-algebras in the sense of extended deformation theory.
 Furthermore, when $(L,A)$ is a matched Lie pair, i.e. the quotient $L/A$ is also a Lie subalgebroid of $L$, we investigate isomorphism classes of infinitesimal deformations modulo automorphisms from exponentials of derivations along the normal direction $L/A$. The extended deformation theory of the associated deformation functor recovers the formal deformation theory of complex structures and that of transversely holomorphic foliations.
\end{abstract}
\maketitle



\tableofcontents
\parskip = 0em 

 \section*{Introduction}

\noindent\textbf{The motivation.}

The ``Deligne principle" in deformation theory posits a fundamental correspondence: every deformation problem can be represented by a deformation functor. This functor, in turn, is governed by a differential graded (dg) Lie algebra (or, more generally, an $L_\infty$-algebra). Specifically, the dg Lie algebra controls the deformations via solutions of the Maurer-Cartan equation  modulo gauge actions \cite{GM}. This approach is now widely adopted and often referred to as formal or extended deformation theory  \cite{ManettiA}.

 This paper aims to initiate the study of an extended deformation theory for \textbf{Lie algebroid pairs}, which we refer to as Lie pairs for brevity. To contextualize this, recall that a Lie algebroid over $\K$ (where $\K$ represents either the field of real numbers, $\R$, or complex numbers, $\C$) is a $\K$-vector bundle $L$ over a base manifold $M$. This bundle is endowed with a Lie bracket $[-,-]$ defined on its sections, along with a bundle map $\rho \colon L \to T_M \otimes_\R \K =: \TMK$, called the anchor. The anchor map $\rho$ must satisfy two conditions: it acts as a morphism of Lie algebras on the section spaces, and it adheres to the Leibniz identity:
 \[
 [X,fY] = f[X,Y] + (\rho(X)f) Y,
 \]
 for all $X,Y \in \Gamma(L)$ and $f \in C^\infty(M,\K)$.
 A \textit{Lie (algebroid) pair} is then defined as an inclusion $A\hookrightarrow L$ of Lie algebroids sharing a common base space, denoted by $(L,A)$. These Lie pairs arise naturally in diverse mathematical domains, including Lie theory, complex geometry, foliation theory, and Poisson geometry. For instance, a complex manifold $X$ gives rise to the Lie pair $(T_X \otimes \C, T_X^{0,1})$ over $\C$. Similarly, a regular foliation $\mathcal{F}$ on $M$ defines a Lie pair $(T_M, F)$ over $\R$, where $F \subset T_M$ represents the integrable distribution tangent to the foliation $\mathcal{F}$.  The Molino class of a foliation $\mathcal{F} \subset M$ and the Atiyah class of a complex manifold $X$ can be interpreted as the Atiyah classes of their corresponding Lie pairs~\cite{CMP16}. Furthermore, the Atiyah class of any Lie pair $(L,A)$ induces an $L_\infty$-algebra structure on the shifted tangent complex $\Gamma(\Lambda^\bullet A^\vee \otimes L/A[1])$~\cite{Camille-S-X-2021}. In the specific case where the Lie pair originates from a compact K\"{a}hler manifold, this induced $L_\infty$-algebra structure recovers the fundamental construction in Kapranov's formulation of Rozansky-Witten theory~\cite{Kap}.

 Inspired by the success of formal deformation theory for complex structures~\cite{Manetti} and regular foliations~\cite{Heitsch} in differential geometry, we initiate a study of infinitesimal deformations of Lie pairs. Specifically, our aim is to classify infinitesimal deformations of a Lie pair $(L, A)$ up to a suitably defined notion of isomorphism. We anticipate that this classification is equivalent to identifying gauge relations between Maurer-Cartan elements within certain $L_\infty$-algebras.

 ~\\
\noindent\textbf{The main results.}

We outline the contents and main results of this paper. First in Section \ref{Sec:Aringedobject}, we introduce the concept of infinitesimal thickenings of Lie algebroids, a tool specifically designed to describe infinitesimal deformations of Lie pairs.   The notion of an $\v$-ringed manifold $M_{\v}$ refers to a smooth manifold $M$ enriched with a local Artinian $\K$-algebra $\v$. More precisely, it is a locally ringed space over $M$ whose structure sheaf is the sheaf of smooth functions thickened by $\v$ (see Definition~\ref{Def: ringed manifolds}).

Within the category of $\v$-ringed manifolds, we identify vector bundle objects, which we call $\v$-ringed vector bundles (see Definition~\ref{Def: ringed vector bundles}). An $\v$-ringed vector bundle can be interpreted as an infinitesimal thickening of a conventional vector bundle, achieved through the application of the local Artinian $\K$-algebra $\v$. Analogously, we define $\v$-ringed Lie algebroids (see Definition~\ref{Def: ringed Lie algebroids}) as Lie algebroid objects within the category of $\v$-ringed manifolds, representing infinitesimal thickenings of standard Lie algebroids. Given an $\v$-ringed Lie algebroid $L_\v$, evaluating it at the unique maximal ideal of $\v$ yields a Lie algebroid $L$, which we designate as the center Lie algebroid of $L_\v$.

Second, in Section \ref{sec:infinitesimaldeform},  we define an infinitesimal deformation of a Lie pair $(L, A)$ as a Lie pair $(L^0_\v, A_\v)$ within the category of $\v$-ringed manifolds, centered around the original Lie pair $(L, A)$. Here, $L^0_\v$ denotes the $\v$-linear extension of the Lie algebroid $L$, specifically what we term the $\v$-Cartesian extension of $L$ (refer to Definition~\ref{Def:infdeformation}). The fundamental concept underlying this definition is that the  Lie subalgebroid $A_\v$ is subject to variation, parameterized by $\v$, while the extended Lie algebroid $L^0_\v$ remains stable with respect to $\v$.

 We also need to define isomorphisms between infinitesimal deformations. To   this end, we introduce a specific type of automorphism of $\v$-ringed Lie algebroids   called a ``small automorphism." It is characterized by inducing the identity map on the associated center Lie algebroid, obtained via evaluation at the maximal ideal of $\v$. An interesting fact is that any small automorphism of the $\v$-Cartesian extension $L_\v^0$ can be represented as the exponential $\exp(\delta)$ of a nilpotent derivation $\delta$ of $L_\v^0$ (see Proposition~\ref{prop-aut}). Consequently, the small automorphism group $\saut(L_\v^0)$ possesses a subgroup, denoted $\operatorname{sIAut}(L_\v^0)$, comprising exponential elements derived from nilpotent inner derivations of $L_\v^0$. Using these concepts, we define weak and semistrict isomorphisms. Two infinitesimal deformations, $(L^0_\v, A_\v)$ and $(L^0_\v, A^\prime_\v)$, of $(L,A)$ are defined as weak (respectively, semistrict) isomorphic if there exists a small morphism $\Pi_A \colon A^\prime_\v \to A_\v$ of $\v$-ringed Lie algebroids and a small automorphism $\exp(\delta)$, belonging to $\saut(L_\v^0)$ (respectively, $\operatorname{sIAut}(L_\v^0)$), that relates the two deformations in a proper manner. For a rigorous definition, refer to Definition~\ref{Def:full-and-semi-iso}.

  Given a Lie pair $(L, A)$, assigning its weak or semistrict isomorphic infinitesimal deformations to each local Artinian $\K$-algebra determines two infinitesimal deformation functors, denoted by $\infDefFctLA$ and $\sinfDefFctLA$, respectively, from the category of local Artinian $\K$-algebras to the category of sets. Consequently, it is pertinent to inquire about the $L_\infty$-algebras, denoted as $\h$ and $\h_0$, that govern these deformation functors. Specifically, we seek $\h$ and $\h_0$ such that the associated algebraic deformation functors ${\rm Def}_{\h}$ and ${\rm Def}_{\h_0}$ are isomorphic to $\infDefFctLA$ and $\sinfDefFctLA$, respectively. Here, ${\rm Def}_{\h}$ (resp. ${\rm Def}_{\h_0}$) maps each local Artinian $\K$-algebra $\v$ to the set of gauge equivalent classes of Maurer-Cartan elements of the nilpotent $L_\infty$-algebra $\h \otimes \maximalidealofartin$ (resp. $\h_0 \otimes \maximalidealofartin$) in the sense of Getzler~\cite{Getzler} (see also~\cite{Guan}). 

In fact, associated with a Lie pair $(L,A)$, there exists a cubic $L_\infty$-algebra
\[
\basic:= \Gamma(\Lambda^\bullet A^\vee \otimes L/A),
\]
as demonstrated in \cite{BCSX}. This $L_\infty$-algebra, which we call the basic cubic $L_\infty$-algebra of $(L,A)$, differs from the construction presented in \cite{Camille-S-X-2021}. A concrete illustration of this arises when considering the Lie pair $(T_\C X, T_X^{0,1})$ derived from a compact complex manifold $X$. In this specific case, the corresponding basic cubic $L_\infty$-algebra $\basic$ is isomorphic to the Kodaira-Spencer algebra $\Omega_X^{0,\bullet}(T^{1,0}_X)$, which governs the infinitesimal deformations of complex structures on $X$.

The significance of $\basic$ lies in the fact that equivalence classes of Maurer-Cartan elements within the cubic (and nilpotent) $L_\infty$-algebra $\basic \otimes \maximalidealofartin$ are isomorphic to the set of standard deformations of $(L,A)$, as detailed in Proposition~\ref{lem:MCstandard}. However, it is important to note that the degree $0$ component of $\basic$, specifically $\Gamma(L/A)$, is insufficient to generate the small automorphism group $\saut(L_\v^0)$ via the $L_\infty$ exponential map. Furthermore, there are numerous examples involving general Lie algebra pairs that show the algebraic deformation functor associated with $\basic$ is neither isomorphic to the infinitesimal deformation functor ${\infDefFctLA}$ nor to $\sinfDefFctLA$.

 To address this problem, we couple the Lie algebra of derivations of the Lie algebroid $L$, denoted by $\Der(L)$, with the cubic $L_\infty$-algebra $\basic$. As demonstrated in \cite{NCCH}, there exists a natural action of $\Der(L)$ on $\basic$. Consequently, the direct sum
\[
   \h:=\Der(L) \oplus \basic,
\]
inherits a cubic $L_\infty$-algebra structure that extends the canonical structure on $\basic$. We call  $\h$ the extended cubic $L_\infty$-algebra of the Lie pair $(L,A)$. Furthermore, $\h$ contains an $L_\infty$-subalgebra given by
\[
 \h_0 =\operatorname{IDer}(L) \oplus \basic,
\]
where $\operatorname{IDer}(L) \subset \Der(L)$ represents the Lie subalgebra of inner derivations of $L$. These two cubic $L_\infty$-algebras, $\h$ and $\h_0$, are central to our investigation. Our main result, detailed in Theorems~\ref{Thm:MainTheorem} and ~\ref{Thm: semistrict}, can be summarized as follows:

\begin{theorem}
 The infinitesimal deformation functor ${\infDefFctLA}$ of weak isomorphisms classes is controlled by the extended cubic $L_\infty$-algebra $\h$, while $\sinfDefFctLA$ of semistrict isomorphism classes is controlled by the $L_\infty$-subalgebra $\h_0$ of $\h$.
\end{theorem}

To illustrate the application of this theorem, let us consider deformations of a Lie algebra pair $\a \subset \mathfrak{l}$. In this specific case, the base manifold $M$ reduces to a single point. Consequently, the tangent space of the functor $\operatorname{sDef}_{(\mathfrak{l},\a)}$ is isomorphic to the deformation space of the Lie subalgebra $\a$ within $\mathfrak{l}$, as defined by Crainic, Schätz, and Struchiner in \cite{Crainic2014}.

 In the third part of this paper, Section \ref{Sec:thirdpartmatchedpair}, we investigate the specific scenario where $L = A \bowtie (L/A)$ constitutes a matched Lie pair, implying that $L/A$ can be embedded into $L$ as a Lie subalgebroid. Consequently, the set $\{\adLb \mid b \in \Gamma(L/A)\}$ of inner derivations along the ``normal direction'' $L/A$ forms a Lie subalgebra within the inner derivations of $L$. The exponential of these derivations, defined as
\[
  \operatorname{hAut}(L_\v^0) = \{\exp(\adLb) \mid b \in \Gamma(L/A) \otimes \maximalidealofartin\},
\]
forms a subgroup of the small automorphism group $\saut(L_\v^0)$. In this context, we employ $\operatorname{hAut}(L_\v^0)$ to define isomorphism classes of infinitesimal deformations, as detailed in Definition \ref{Def:half-iso}. The functor corresponding to this relation is denoted by $ \operatorname{hDef}_{A \bowtie B}$. Furthermore, the basic cubic $L_\infty$-algebra $\basic$ associated with the matched Lie pair simplifies to a dg Lie algebra. This dg Lie algebra governs the aforementioned isomorphism classes of infinitesimal deformations, as formalized by the following statement (see Theorem~\ref{Thm: matched pair} for more details):

\begin{theorem}
For a matched Lie pair $L = A \bowtie (L/A)$, the infinitesimal deformation functor $ \operatorname{hDef}_{A \bowtie B}$ is isomorphic to the algebraic deformation functor associated with the dg Lie algebra $\basic$.
\end{theorem}

 As an application, we consider the matched Lie pair $T_X^{0,1} \bowtie T_X^{1,0}$ on a complex manifold $X$. The associated infinitesimal deformation functor $\operatorname{hDef}_{T_X^{0,1} \bowtie T_X^{1,0}}$ is isomorphic to the infinitesimal deformation functor of complex structures on $X$. Furthermore, the basic dg Lie algebra $\basic$ corresponds to the Kodaira-Spencer algebra of $X$. Consequently, this recovers the established result regarding the isomorphism between the functor of infinitesimal deformations of complex structures and the algebraic deformation functor associated with the Kodaira-Spencer algebra.

We also extend our investigation to transversely holomorphic foliations $\mathcal{F}$ on  compact smooth manifolds $M$. Let $F$ denote the tangent bundle of $\mathcal{F}$, and let $B = T_M/F$ represent the normal bundle. This normal bundle $B$ possesses a natural complex structure, inducing a splitting $B^\C = B^{1,0} \oplus B^{0,1}$ of its complexified bundle $B^\C$. We then construct a matched Lie pair $(F^\C \oplus B^{0,1}) \bowtie B^{1,0}$. The deformation functor associated with this matched Lie pair is isomorphic to the deformation functor of the transversely holomorphic foliation $\mathcal{F}$ itself \cites{Spencer1,Spencer2, DK,Gomez}.

~\\
\noindent\textbf{Related works.}

Several works in the literature relate to the deformation theory presented in this paper. The study of Lie algebra deformations originates with the work of Nijenhuis and Richardson~\cites{Nijenhuis-Richardson, Richardson}.
For a more recent overview of Lie algebra pair deformations, see Crainic~\cite{Crainic2014}. Ji~\cite{jixiang} investigated the simultaneous deformations of Lie algebroids and their Lie subalgebroids, deriving an $L_\infty$-algebra via higher derived brackets. However, this approach differs significantly from the Lie pair deformation framework presented here. Specifically, our Lie pair setting $(L, A)$ considers a Lie algebroid $L$ and its Lie subalgebroid $A$ defined over the ``same'' base manifold. This choice is motivated by examples arising from foliations and complex manifolds. In contrast, Ji's framework typically involves a Lie algebroid $E$ over a smooth manifold $M$, where the base manifold of its Lie subalgebroid is a submanifold of $M$. This is motivated by examples from Poisson geometry involving coisotropic submanifolds of Poisson manifolds. In subsequent work, Ji \cite{jixiang2017JGP} further explored the relationship between deformations of Lie subalgebroids and deformations of coisotropic submanifolds (see also~\cites{Cattaneo, Oh} on deformations of coisotropic submanifolds).

~\\
\noindent\textbf{Acknowledgment.} 

We would like to thank Ping Xu for fruitful discussions.

~\\
\noindent\textbf{List of commonly used notations.}

\begin{enumerate}
\item  $\K$  ---    the field $\R$ of real numbers, or the field $\C$ of complex numbers;
  \item  $\maximalidealofartin$  ---   the maximal ideal of a local Artinian $\K$-algebra $\v$;
  \item  $M_\v=(M, \O_{M_\v})$  ---   an $\v$-ringed manifold centered on a smooth manifold $M$;
   \item $\Mf_{\v}$  --- the category  of $\v$-ringed manifolds;
  \item  $ \saut (M_\v)$  ---  the group of small automorphisms of an $\v$-ringed manifold $M_\v$;
  \item $\Gamma (\TMK )$ --- the space of $\K$-valued vector fields on $M$, i.e., $\K\otimes_\R \Gamma(T_M)$;
 \item $ T_{\v} := (\TMK )_{\v} $ --- the $\v$-ringed vector bundle with center $\TMK\to M$;
 \item $\Gamma(T_\v):=\Gamma(T_M)\otimes_{\mathbb{R}} \v$ --- the space of global sections of $ T_{\v}$;
  \item $\K_2[t]$ --- the $\K$-algebra of dual numbers ;
  \item  $(E_\v, [-,-]_{E_\v} ,\rho_{E_\v})$  ---  an  $\v$-ringed Lie algebroid;
  \item $E_\v^0$ ---  the $\v$-Cartesian extension of a Lie algebroid $E$;
  \item $\saut(E_{\v})$ --- the group of small automorphisms of  an $\v$-ringed Lie algebroid $E_\v$;
  \item $\Der(E)$ --- the space of derivations of a Lie algebroid $E$;
  \item $\operatorname{IDer}(E)$ --- the space of inner derivations of a Lie algebroid $E$;
  \item $\operatorname{sIAut}(E_\v^0)$ --- the group of small inner automorphisms of    $E_\v^0$ (see (11));
  \item  ${\rm Sd}(L,A,\v)$ --- the set of solutions $\xi \in \Gamma(A^* \otimes B) \otimes \maximalidealofartin$  to Equation~\eqref{Eqt:standarddeformation};
      \item $\MC(\g)$ --- the set of Maurer-Cartan elements of a nilpotent $L_\infty$-algebra $\g$;
  \item $\Omega_A^\bullet(B)=\Gamma(\Lambda^\bullet A^\vee \otimes B)$ ---  the basic cubic $L_\infty$-algebra arising from the Lie pair $(L,A)$;
  \item $\h = (\Der(L) \oplus \Gamma(B)) \bigoplus \left(\bigoplus_{n \geqslant 1} \Omega_A^n(B)\right)$ ---  the extended cubic $L_\infty$-algebra arising from the Lie pair $(L,A)$;
  \item $\firstDef$ --- the algebraic deformation functor associated to an $L_\infty$-algebra $\g$;
  \item $\infDefFctLA$ --- the {weak infinitesimal deformation functor} of the Lie pair $(L,A)$;
  \item$\ArtCat$ --- the category of local Artinian $\K$-algebras;
  \item \textbf{Set} --- the category  of sets;
  \item $\sinfDefFctLA$  ---  the semistrict infinitesimal deformation functor of the Lie pair $(L,A)$;
  \item $\operatorname{hDef}_{A \bowtie B}$ ---  the deformation functor of the matched Lie pair $L = A \bowtie B$;
      \item   $\H^i (\g, [-]_1)$  --- the $i$-th cohomology of an $L_\infty$-algebra $\g$ with respect to the first bracket $[-]_1$.
\end{enumerate}

\section{Lie algebroids enriched with local Artinian algebras}\label{Sec:Aringedobject}

We begin by defining smooth manifolds enriched with a local Artinian $\K$-algebra $\v$.

\subsection{Local Artinian ringed manifolds and vector bundles}
Given a smooth manifold $M$, we denote by $\mathcal{O}_{M}$ the sheaf of $\K$-valued smooth functions on $M$ and by $C^\infty(M,\K)$ the $\K$-algebra of smooth functions on $M$.

\begin{Def}\label{Def: ringed manifolds}
Let $M$ be a smooth manifold, and $\v$ a local Artinian $\K$-algebra. Define a sheaf $\O_{M_\v}$ of $\v$-algebras over $M$  by assigning to every open subset $U \subset M$  the $\v$-algebra $$ \mathcal{O}_{M_\v} (U) := C^\infty(U,\K) {\otimes}_{\K} \v . $$
	We call $M_\v:=(M, \O_{M_\v})$ the $\v$-ringed manifold (centered on $M$).
The space of global sections of   $\O_{M_\v}$ is denoted by  $C^\infty(M_\v) (= \CinfMK \otimes_{\K} \v)$; its  elements are called smooth functions on $M_{\v}$.
\end{Def}

Let $\mathcal{E}$ be a sheaf of $\K$-modules over $M$.
The \emph{evaluation map} $ \ev \colon \v \to \v / \maximalidealofartin = \K$ at the maximal ideal $\maximalidealofartin$ of $\v$ induces a map $\mathcal{E}\otimes_\K \v \to \mathcal{E}$, also denoted by $\ev$, sending $\xi \otimes a $ to $ \xi \cdot \ev(a)$ for all local sections $\xi$ and all $a \in \v$.
In particular, when $\mathcal{E}$ is the structural sheaf $\O_{M_\v}$, the evaluation map induces an embedding from the $\K$-ringed manifold  $(M, \O_M)$  to the $\v$-ringed manifold $ (M, \O_{M_{\v}})$.

\begin{Def}\label{Def:morphism}
A morphism of $\v$-ringed manifolds $M_{\v} \to N_{\v}$ consists of a pair $(\varphi, \lambda^{\sharp})$ where
\begin{enumerate}
\item $\varphi \colon M \to N$ is a smooth map of smooth manifolds;
\item $\lambda^\sharp \colon \O_{N_{\v}} \to   \varphi_{*} \O_{M_\v}$ is a morphism of sheaves of $\v$-algebras over $ N $ extending the pullback map $\varphi^* \colon \O_{N} \to   \varphi_{*} \O_{M}$ in the following sense --- for each local section $  f \otimes a  $ of $  \O_{N_{\v}}(U)=$ $C^\infty(U,\K){\otimes}_{\K} \v$, where $U\subset N$ is open, one has
\[
  \ev \circ \lambda^{\sharp}( f \otimes a )  = \ev \circ (\lambda^\sharp(f) \otimes a )=\varphi^*(f) \cdot \ev(a).
\]
\end{enumerate}
\end{Def}

 Since the restriction map to each stalk at $x \in M$, $\lambda_x^{\sharp} \colon \O_{N_{\v}, \varphi(x)} \to \O_{M_{\v}, x}$ is local, the morphism $(\varphi, \lambda^{\sharp})$  of $\v$-ringed manifolds is indeed a morphism of locally ringed spaces.
 Meanwhile, the morphism $ \lambda^\sharp \colon  \O_{N_{\v}} \to \varphi_*\O_{M_\v}$ of sheaves is completely determined by the corresponding morphism on the level of global sections,
\[
\lambda^* \colon C^\infty(N_\v)  \to C^\infty(M_\v).
\]
So we can regard a morphism of $\v$-ringed manifolds from $M_{\v}$ to $N_{\v}$ as a morphism of $\v$-algebras $\lambda^*\colon C^\infty(N_\v)  \to C^\infty(M_\v)$ covering a smooth map $\varphi \colon M \to N$ in the sense that
\begin{equation*}
 \ev\circ \lambda^* = \varphi^\ast \colon C^\infty(N,\K)  \to  C^\infty(M,\K).
\end{equation*}
In this situation, we  say that $\lambda^*$ is a \textbf{lifting} of $\varphi^*$ and that $\varphi^*$ is the \textbf{center} of $\lambda^*$.

It is clear that the collection of $\v$-ringed manifolds and their morphisms forms a category, denoted by $\Mf_{\v}$.
We need a particular type of automorphisms of $\v$-ringed manifolds.
 \begin{Def}
 	A  \textbf{small automorphism} of an $\v$-ringed manifold $M_\v$ is a lifting of the identity map $\id$ on $\CinfMK$, i.e., a morphism  $ \lambda^* \colon C^\infty(M_{\v}) \to C^\infty(M_\v)$ satisfying $\ev\circ \lambda^*=\id$.
\end{Def}
We denote by $ \saut (M_\v)$  the group of small automorphisms of $M_\v$.
\begin{Ex}\label{Ex:Kvarepsilon1}
    Let $\v = \K_2[t]:= \K[t]/(t^{2})(\cong \K\oplus \K t)$ be the algebra of dual numbers. Then for any $X   \in \Gamma (\TMK ) :=  \K\otimes_\R \Gamma(T_M) $, the exponential map
    \[
      \exp {(X\otimes t)}=\id + X \otimes t\colon \CinfMK \otimes_{\K} \K_2[t] \to \CinfMK \otimes_{\K} \K_2[t],
    \]
    which sends $f  \in \CinfMK$ to $ f   +  X(f)\otimes t $, is an element in $ \saut (M_{\K_2[t]})$.
    It is clear that the inverse of $ \exp(X\otimes t)$ is $ \exp (-X\otimes t)$.
\end{Ex}
\begin{Rem}
In fact, any small automorphism of an $\v$-ringed manifold $M_\v$ is of the form $\exp (D)$ for some $D\in \Gamma (\TMK) \otimes_\K \maximalidealofartin $. Here $\exp$ denotes the \textit{exponential map}:
 \[
 \exp \colon  \Gamma (\TMK) \otimes_\K \maximalidealofartin \to \saut (M_\v),\quad D \mapsto \exp(D) ,
 \]
 where $\exp (D)$ is defined by
\[
  \exp (D)(f) := \sum_{n\geq 0} \frac{1}{n!} \langle D^n, f  \rangle,
\]
for all $f \in C^{\infty}(M_\v)$.  The bracket $\langle  D^n, f\rangle$ means the $\v$-linear action of the  differential operator $D^n$ on $f\in C^\infty(M_\v)$. The above sum is indeed finite as $D \in \Gamma (\TMK )\otimes_\K \maximalidealofartin$ is nilpotent. (See Proposition \ref{prop-aut}.)
\end{Rem}

We now turn to vector bundle objects in the category of $\v$-ringed manifolds.
\begin{Def}\label{Def: ringed vector bundles}
An \textbf{$\v$-ringed vector bundle} over an $\v$-ringed manifold $M_{\v}$ is an $\v$-ringed manifold  $E_{\v}$ together with a morphism $E_{\v}\to M_{\v}$, which covers a smooth $\K$-vector bundle $\pi\colon E\to M$.

The underlying vector bundle $ E\to M $ is called the \textbf{center} of $ E_{\v} \to M_{\v} $.
The space of global sections of $E_{\v}$ is defined by
 \[
   \Gamma(E_\v):=\Gamma(E) \otimes_{\K} \v,
  \]
 which is a $C^{\infty}(M_\v)$-module.
 \end{Def}

\begin{Def}\label{Def:comorphismAringedvectorbundle}
A morphism of $\v$-ringed vector bundles  from $F_{\v} $ to $ E_{\v} $ (over the same base $M_\v$) is specified by an $\v$-linear map $\Pi\colon \Gamma(F_\v)\to \Gamma(E_\v)$ satisfying the following two conditions:
  	\begin{enumerate}
  \item The map $\Pi$ covers a small automorphism $\lambda^*$ of $M_\v$ in the sense that
\[
\Pi(f \nu)=\lambda^*(f)\Pi(\nu),
\]
for all $f \in C^\infty(M_\v), \nu\in \Gamma(F_\v)$.
\item The center of $\Pi$: $$\Pi_0:= \ev \circ \Pi \colon \Gamma(F) \to \Gamma(E)$$  defines a morphism of vector bundles over $M$.
  	\end{enumerate}
Such a morphism will be denoted by $(\Pi, \lambda^\ast)$.
In particular, when the small automorphism $\lambda^\ast$ of $C^\infty(M_\v)$ is the identity $\id$, we call  $(\Pi, \id )$ a \textbf{strict morphism} and denote it by $\Pi$ for simplicity.
\end{Def}
Compositions of morphisms of $\v$-ringed vector bundles are naturally defined.

\begin{Rem}
In general, morphisms between $\v$-ringed vector bundles over distinct $\v$-ringed base manifolds are more complicated than what we have discussed here.
\end{Rem}

\begin{prop}\label{prop: Isom of vector bundles}
Suppose that $(\Pi, \lambda^\ast) \colon F_\v \to E_\v$ is a morphism of $\v$-ringed vector bundles over $M_\v$. Then the map $ \Pi  \colon \Gamma(F_\v) \to \Gamma(E_\v)$  is injective  (resp. surjective) if and only if the center $\Pi_0 \colon \Gamma(F) \to \Gamma(E)$ of $\Pi$ is injective (resp. surjective). Consequently, $\Pi$ is an isomorphism of $\v$-modules if and only if its center $\Pi_0$  is an isomorphism of vector bundles.
\end{prop}
\begin{proof}
Assume that $\m^n_\v \neq 0$ while $\m^{n+1}_\v=0$ for some $n\geqslant 1$.
We now show that $\Pi$ is injective if and only if $\Pi_0 \colon \Gamma(F) \to \Gamma(E)$ is injective.

Assume first that $\Pi$ is injective. If $\Pi_0(e)  = 0$ for some $e \in \Gamma(F)$,
 it then follows from $ \ev(\Pi(e))  = \Pi_0(e)  = 0$ that $\Pi(e)  \in  \Gamma(E)  \otimes_{\K} \maximalidealofartin$. Thus, for any nonzero element $a \in \m^n_\v$, we have $\Pi(e \otimes a) = \Pi(e) a = 0$, which implies that $e \otimes a = 0$ since $\Pi$ is injective. Hence, we have $e = 0$, and thus $\Pi_0$ is injective.

Conversely, assume that  $\Pi_0 \colon \Gamma(F) \to \Gamma(E)$ is injective.  We next show that $\Pi$ is injective as well.
Consider the following family of $\v$-linear maps induced by $\Pi$,
 \[
  \Pi_k \colon \Gamma(F_{\v/\maximalidealofartin^{k+1}})\to \Gamma(E_{\v/\maximalidealofartin^{k+1}}),
  \]
for all $0 \leqslant k \leqslant n $, where $\Pi_{n }=\Pi$.
By an induction argument on $k$, it suffices to prove that $\Pi_{k}$ is injective provided that $\Pi_{k-1} $ is injective for all $k  $.

By choosing a splitting of the following short exact sequence of finite dimensional vector spaces
\[
0 \to \maximalidealofartin^k/ \maximalidealofartin^{k+1} \to \v/\maximalidealofartin^{k+1} \to \v/\maximalidealofartin^{k} \to 0,
\]
we obtain a direct sum decomposition of $\K$-linear vector spaces $\v /\maximalidealofartin^{k+1} \cong  \v /\maximalidealofartin^k \oplus (\maximalidealofartin^k/\maximalidealofartin^{k+1})$.
Thus, we may decompose the map $\Pi_{k }$ as follows:
\[
\Pi_{k } = \Pi_{k-1} + \phi_{k },
\]
where the map $ \phi_{k }$ is the component $ \Gamma(F_{\v/\maximalidealofartin^{k+1}}) \to \Gamma(E_{\maximalidealofartin^k/ \maximalidealofartin^{k+1}}) $ of $\Pi_{k }$.
For any
\[
\sum_i s_i\otimes a_i+\sum_j t_j\otimes b_j \in \Gamma(F_{\v/\maximalidealofartin^{k}}) \oplus \Gamma(F_{\maximalidealofartin^k/\maximalidealofartin^{k+1}}) \cong   \Gamma(F_{\v/\maximalidealofartin^{k+1}}),
\]
where $s_i, t_j\in\Gamma(F)$ and $a_i\in \v /\maximalidealofartin^k,  b_j \in \maximalidealofartin^k/\maximalidealofartin^{k+1}$, we have
\begin{align*}
\Pi_{k }\left(\sum_i s_i\otimes a_i+\sum_j t_j\otimes b_j\right) &= \Pi_{k-1}\left( \sum_i s_i\otimes a_i \right) +\phi_{k }\left( \sum_j t_j\otimes b_j \right) \\
&= \Pi_{k-1}\left(\sum_is_i\otimes a_i\right) + \sum_j \Pi_0(t_j)\otimes b_j,
\end{align*}
which implies that $\Pi_{k }$ is injective.

To prove that $\Pi$ is surjective if and only if $\Pi_0$ is surjective, it suffices to adopt an analogous approach, which we omit.
\end{proof}

\begin{Ex}[Tangent bundle of an $\v$-ringed manifold]\label{Ex:tangent-bundle}
Consider the associated $\v$-ringed tangent bundle $ T_{\v} := (\TMK )_{\v} $ over the $ \v$-ringed manifold $ M_\v$.
Each small automorphism $\lambda^\ast$ of $M_\v$ induces an automorphism $(\Pi_{\lambda^\ast}, \lambda^\ast)$ of $T_\v$ by conjugation, i.e.,
\[
\Pi_{\lambda^\ast}(D) = \lambda^\ast \circ D \circ (\lambda^{\ast})^{-1},
\]
for all $D \in \Gamma(T_\v)=\Gamma(T_M)\otimes_{\mathbb{R}} \v$.

In particular, when $\v = \K_2[t]$ is the algebra of dual numbers, by Example \ref{Ex:Kvarepsilon1}, each vector field  $D_0 \in \Gamma (\TMK )$ induces a small automorphism $\exp {(D_0 \otimes t)} = \id + D_0 \otimes t$ of $M_\v$.
In this case, the associated morphism is
\begin{align}\label{Ex:Tangentbundle}
 \Pi_{\exp {(D_0 \otimes t)}}(D) &= \exp {(D_0 \otimes t)} \circ D \circ \exp {(-D_0 \otimes t)}  = (\id + D_0 \otimes t) \circ D \circ (\id - D_0 \otimes t) \notag \\
 &= D + [D_0, D]\otimes t,
\end{align}
for all $D \in \Gamma(T_\v)$.
\end{Ex}

\subsection{Local Artinian ringed Lie algebroids}
We now study Lie algebroid objects in the category of $\v$-ringed manifolds.
\begin{Def}\label{Def: ringed Lie algebroids}
An \textbf{$\v$-ringed Lie algebroid} consists of a triple $(E_\v, [-,-]_{E_\v}, \rho_{E_\v})$, where
\begin{compactenum}
  \item $E_\v$ is an $\v$-ringed vector bundle over $M_\v$;
  \item $\rho_{E_\v}$, called the anchor, is a morphism of $\v$-ringed vector bundles from $E_\v$ to the tangent bundle $T_\v$ of $M_\v$,  covering the identity $\id  \colon M_\v \to M_\v$;
  \item $[-,-]_{E_\v} \colon \Gamma(E_\v) \times \Gamma(E_\v) \to \Gamma(E_\v)$ is an $\v$-bilinear Lie bracket on the space $\Gamma (E_\v)$ of global sections, satisfying the Leibniz rule
\[
[u, f   v]_{E_\v}  = \langle \rho_{E_\v} (u), f  \rangle   v + f   [u,v]_{E_\v} ,
\]
for all $f  \in  C^\infty(M_\v)$  and $u,v\in \Gamma(E_\v)$.
\end{compactenum}
\end{Def}

 Given an $\v$-ringed Lie algebroid $(E_\v, [-,-]_{E_\v} ,\rho_{E_\v})$,  the evaluation map $\ev \colon E_\v \to E$ determines a Lie algebroid $(E, [-,-]_E ,\rho_E)$. 
Conversely, given a Lie algebroid $ (E, [-, -]_E,\rho_E) $ over a smooth manifold $M$, there exists an $\v$-ringed Lie algebroid, denote by $E^0_\v$, whose anchor $\rho_{E_\v^0}$ and bracket $[-,-]_{E_\v^0}$ are $\v$-linear extensions of the anchor $\rho_E$ and the bracket $[-,-]_E$, respectively.
We call $E_\v^0$ the $\v$-\textbf{Cartesian extension} of the Lie algebroid $E$.

\begin{Def}
Let $(E_\v, [-,-]_{E_\v}, \rho_{E_\v})$ and $(F_\v, [-,-]_{F_\v}, \rho_{F_\v})$ be two $\v$-ringed Lie algebroids over  $M_\v$.  A morphism of  from $F_\v$ to $E_\v$  is a morphism $(\Pi, \lambda^\ast) \colon F_\v \to E_\v$ of the underlying $\v$-ringed vector bundles satisfying the following two conditions:
\begin{compactenum}
\item The $\v$-linear map $\Pi \colon \Gamma(F_\v) \to \Gamma(E_\v)$ preserves the brackets, i.e.,
\[
\Pi([s_1,s_2]_{F_\v})=[\Pi(s_1),\Pi(s_2)]_{E_\v},
\]
for all $s_1, s_2\in\Gamma(F_\v)$.
\item The pair $(\Pi,\lambda^\ast)$ is compatible with the two anchors $\rho_{E_\v}$ and $\rho_{F_\v}$ in the following sense:
\begin{equation}\label{Eq: comor of LA}
\lambda^\ast \langle \rho_{F_\v} (s), f\rangle=\langle \rho_{E_\v}(\Pi(s)), \lambda^\ast f\rangle,
\end{equation}
for all $f \in C^{\infty}(M_{\v})$ and $s \in\Gamma(F_\v)$.
\end{compactenum}
\end{Def}
It is straightforward to verify that the center $\Pi_0 \colon F \to E$ of a morphism $\Pi$ of $\v$-ringed Lie algebroids is itself a morphism of Lie algebroids over $M$.
In particular, when $\lambda^\ast$ is the identity $\id$ of $C^\infty(M_\v)$, the morphism $(\Pi, \id)$ will be denoted by $\Pi$ for simplicity and will be referred to as a \textbf{strict} morphism.

\begin{Rem}
Given two Lie algebroids over different base manifolds, Liu and Chen have studied various characterizations of morphisms and comorphisms between them in~\cite{JA2007}. The  definition of morphisms and comorphisms of ordinary Lie algebroids can be generalized to the setting of $\v$-ringed Lie algebroids. In fact, our definition of morphisms of $\v$-ringed Lie algebroids over the same $\v$-ringed manifold can be viewed as a special case of \emph{comorphisms} of $\v$-ringed Lie algebroids in the spirit of~\cite{JA2007}.
\end{Rem}

In what follows, we focus on a special kind of automorphisms of an $\v$-ringed Lie algebroid.
\begin{Def}\label{Def:smallcoautoALiealgebroid}
 A \textbf{small automorphism} of an $\v$-ringed Lie algebroid $(E_\v, [-,-]_{E_\v}, \rho_{E_\v})$ over $M_\v$ is a morphism $(\Pi, \lambda^\ast)$ of $\v$-ringed Lie algebroids from $E_\v$ to itself whose center $\Pi_0$ is the identity of the center  Lie algebroid $E$ over $M$.

 Denote by $\saut(E_{\v})$ the group of small automorphisms of  the $\v$-ringed Lie algebroid $E_\v$.
 \end{Def}

For any $\v$-Cartesian extension $E^0_\v$ of a Lie algebroid $E$,
we now establish that every element of $\saut(E^0_{\v})$ can be expressed as the exponential of a nilpotent derivation of $E^0_\v$.

\begin{Def}\label{Def:Liealgebroiddiff}
   A derivation of the Lie algebroid $(E, [-, -]_E, \rho_E)$ over $M$ is a linear operator $\deltap \colon\Gamma(E) \to \Gamma(E)$ equipped with a vector field $\deltas \in \Gamma(T^\K_M)$, called the symbol of $\deltap$, satisfying
\begin{align*}
\deltap (fu)&=\deltas (f)  u+f \deltap(u),\\
[\deltas,  \rho_E(u)]&=\rho_E(\deltap (u)),\\
\deltap [u,v]_E &=[\deltap (u),v]_E +[u,\deltap  (v)]_E ,
\end{align*}
for all $f \in C^\infty(M,\K), u,v \in \Gamma(E)$.
\end{Def}
For a Lie algebroid $E$, the space $\Der(E)$ of derivations together with the standard commutator is a Lie algebra. The subspace $\operatorname{IDer}(E):=\{\operatorname{ad}_u| u\in\Gamma(E) \}$ of inner derivations is a Lie subalgebra of $\Der(E)$.

Note that the $\v$-linear extension of derivations of a Lie algebroid $E$ yields derivations of the $\v$-Cartesian extension $E^0_\v$ of $E$.  In particular, for any element $\delta \in \Der (E) \otimes_{\K} \maximalidealofartin$, the operator $\delta$ is a nilpotent derivation of $E^0_\v$. Its exponential
\[
\exp(\delta):= \id + \delta + \frac{\delta^2}{2!} + \frac{\delta^3}{3!} + \cdots
\]
defines a small automorphism of $E_\v^0$.
Indeed, every small automorphism of $E_\v^0$ arises in this manner.
To illustrate this fact, consider the special case where $\v = \K_2[t]$ is the algebra of dual numbers over $\K$.
 \begin{Ex}\label{Ex:automorphismLDelta}
Consider the $\K_2[t]$-Cartesian extension $E^0_{\K_2[t]}$ of a Lie algebroid $(E, [\cdot,\cdot]_E, \rho_E)$. Suppose that $\Pi \in \saut(E_{\K_2[t]}) $ is a small automorphism covering a small automorphism $\exp(D\otimes t) = \id +  D\otimes t$ of $M_\v$ for some $D \in \Gamma(\TMK )$ (see Example \ref{Ex:Kvarepsilon1}). Then
\[
 \Pi =\id + \Pi_1\otimes t \colon \Gamma(E)\to \Gamma(E)\otimes_{\K} \K_2[t],
\]
where  $\Pi_1\colon \Gamma(E)\to \Gamma(E)$ is $\K$-linear, satisfying
\[
\Pi_1(fe)=D(f) e +f\Pi_1(e).
\]
Since $\Pi([u,v]_E)=[\Pi(u),\Pi(v)]_E$ for all $u, v \in \Gamma(E)$,  it follows that
\[
	 \Pi_1[u,v]_E=[\Pi_1(u), v ]_E + [u,\Pi_1(v) ]_E.
\]
By Equations~\eqref{Ex:Tangentbundle} and~\eqref{Eq: comor of LA},  one obtains
\[
	\rho_E(\Pi_1(e)) =  [D, \rho_E(e)],
\]
and $\Pi_1 $ is  a derivation of $E$ with symbol $D$. Hence, we have $\Pi=\exp(\Pi_1\otimes t)$.	
\end{Ex}

\begin{prop}\label{prop-aut}
Any small automorphism of the $\v$-Cartesian extension $E^0_\v$  can be uniquely expressed as $\exp(\delta)$ for some   $\delta \in \Der (E) \otimes_{\K} \maximalidealofartin$ of $E^0_\v$. Moreover, $\exp(\delta)$ covers the small automorphism $\exp (\deltas)$ of the $\v$-ringed manifold $M_\v$ generated by the symbol $\deltas\in \Gamma(T_\v)$  of $\delta$.
\end{prop}
\begin{proof}
We prove by induction on the dimension $\dim_\K \v := n \geqslant 2$ of the local Artinian $\K$-algebra $\v$.
When $n=2$,  we have $\v \cong \K_2[t]$. By Example \ref{Ex:automorphismLDelta}, any small automorphism $\Pi \colon E_\v  \to E_\v$ is of the form $\Pi = \id +  \Pi_1\otimes t = \exp(\Pi_1\otimes t)$ for some $\Pi_1\in \Der(E)$.

Suppose that the proposition holds for all local Artinian $\K$-algebras $\v^\prime$ with $ \dim_\K \v^\prime \leqslant n-1$ for some integer $n \geqslant 3$ .
Given an $n$-dimensional local Artinian $\K$-algebra $\v$,  for each nonzero element $t \in \m_\v$ such that $t \m_\v=0$, one has a  short exact sequence of $\K$-vector spaces
\[
	0 \to  \K t \to \v \xrightarrow{\operatorname{pr}}  \v^\prime:= \v/\K t \to 0,
\]
where $\v^\prime$ is a local Artinian $\K$-algebra of dimension $(n-1)$, and $\operatorname{pr}$ is a morphism of Artinian algebras.
	
Given any $\Pi \in \saut(E^0_{\v}) $, the composition
\[
\widetilde{\Pi} = (\id\otimes\operatorname{pr}) \circ \Pi \colon \Gamma(E) \to\Gamma(E)\otimes\v^\prime
\]
determines an element $\Pi^\prime:= \widetilde{\Pi} \otimes \id \in \saut(E^0_{\v^\prime})$.
By the induction assumption, there exists an element $\delta^\prime \in \Der(E)\otimes_{\K}\m_{\v^\prime} $ such that $\Pi^\prime =\exp(\delta^\prime)$.

Now we choose an element $\delta_0\in \Der(E)\otimes_{\K}\m_{\v}$ satisfying $(\id\otimes \operatorname{pr}) \circ \delta_0=\delta'$. Consider the small automorphism $\Pi \circ \exp(-\delta_0)$ of $E_\v$, which is subject to the relation
\begin{align*}
(\id\otimes \operatorname{pr})\circ(\Pi\circ \exp(-\delta_0))&=
\bigl( (\id\otimes \operatorname{pr})\circ \Pi\bigr)\circ \bigl( (\id\otimes \operatorname{pr})\circ  \exp(-\delta_0)\bigr)\\
&=\exp(\delta') \exp(-\delta')=\id.
\end{align*}
It follows that $\Pi \circ \exp(-\delta_0)$ sends $\Gamma(E)\otimes_\K \v$ to $\Gamma(E) \otimes_\K  (\K\oplus \K t)$.
Since  $\K\oplus \K t\cong \K_2[t]$, by Example \ref{Ex:automorphismLDelta}, we can find an element $\delta_1 \in \Der(E)\otimes_\K  \K t$ such that  $\Pi \circ \exp(-\delta_0)= \exp( \delta_1) $.
Hence, we have
\[
\Pi  =\exp( \delta_1)\circ \exp( \delta_0)=\exp(  \delta_1+ \delta_0).
\]
To see the uniqueness of $\delta$, we observe that if $\exp(\delta)=\id\otimes \id_\v$, then we have
\[
\delta+\frac{1}{2}\delta^2+\frac{1}{6}\delta^3+\cdots  = 0 \colon \Gamma(E) \to \Gamma(E)\otimes_{\K} \maximalidealofartin.
\]
Since $\delta^n(\Gamma(E)) \subset \Gamma(E)\otimes_{\K}\maximalidealofartin^n$ for all $n>0$, solving the above equation degreewisely, we obtain that $\delta(\Gamma(E))\subset \Gamma(E) \otimes_{\K}\maximalidealofartin^n$  for all $n>0$, which implies that $\delta = 0$.
\end{proof}

A small automorphism $\varphi$ of $E_\v^0$ is called \textbf{inner} if it is of the form $\exp(\delta)$, where $\delta \in \IDer(E) \otimes_{\K}\maximalidealofartin$ is an inner derivation. The set $\operatorname{sIAut}(E_\v^0)$ consisting of all small inner automorphisms forms a subgroup of $\operatorname{sAut}(E_\v^0)$.

\section{Infinitesimal deformations of Lie pairs}\label{sec:infinitesimaldeform}
Let $(L, A)$ denote an inclusion $i \colon A \hookrightarrow L$ of Lie algebroids over a common base manifold $M$. This section examines infinitesimal deformations of such Lie pairs. The definition of these deformations necessarily involves a local Artinian $\K$-algebra, denoted by $\v$, which serves as the parameter space for the deformation.

\subsection{Infinitesimal deformations and their standard realizations}

\begin{Def}\label{Def:infdeformation}
An \textbf{infinitesimal deformation} of a Lie pair $(L,A)$ parameterized by a local Artinian $\K$-algebra $\v$   consists of the following data:
\begin{compactenum}
	\item an $\v$-ringed Lie algebroid $(A_\v, [-,-]_{A_\v} ,\rho_{A_\v})$ whose center   Lie algebroid is the given Lie subalgebroid $A$ of $L$;
	\item a morphism of $\v$-ringed Lie algebroids from $(A_\v, [-,-]_{A_\v} ,\rho_{A_\v})$ to the $\v$-Cartesian extension $L^0_\v$ of $L$,
\[
(I, \lambda^\ast) \colon (A_\v, [-,-]_{A_\v} ,\rho_{A_\v}) \to L^0_\v,
\]
whose center $I_0$ coincides with the given inclusion $i \colon A \hookrightarrow L$ of Lie algebroids.
\end{compactenum}	
\end{Def}
Conceptually, an infinitesimal deformation of $(L,A)$ depicts an $\v$-parameterized family of Lie algebroid structures $(A_\v, [-,-]_{A_\v}, \rho_{A_\v})$ on the vector bundle $A$, while the `bigger' Lie algebroid $L$ (which contains $A_\v$) remains unchanged with respect to $\v$.
We shall denote an infinitesimal deformation of $(L,A)$ by the quadruple  $([-,-]_{A_\v} ,\rho_{A_\v}; I, \lambda^\ast)$ or by $(I,\lambda^\ast)$ for simplicity.

\begin{Def}\label{Def:isomorphicdeformation}
Let $
([-,-]_{A_\v} ,\rho_{A_\v}, I, \lambda^\ast)$  { and }
$([-,-]'_{A_\v} ,\rho'_{A_\v}, I^\prime, \lambda^{\prime \ast})$
 be   infinitesimal deformations of a Lie pair $(L,A)$. They are said to be \textbf{strictly isomorphic} if there exists an $\v$-ringed Lie algebroid morphism
\[
(\Pi_A, \lambda_A^\ast) \colon
 (A_\v, \rho^\prime_{A_\v}, [-, -]^\prime_{A_\v}) \to (A_\v, \rho_{A_\v}, [-, -]_{A_\v})
\]
such that the following diagram of morphisms  of $\v$-ringed Lie algebroids
\[
\begin{tikzcd}
 L^0_\v & L^0_\v \arrow[l, "{\id}"'] \\
    (A_\v, \rho_{A_\v}, [-, -]_{A_\v})\arrow[u,"{(I,\lambda^\ast)}"]     & (A_\v, \rho^\prime_{A_\v}, [-, -]^\prime_{A_\v}) \arrow[l,"{(\Pi_A,\lambda_A^\ast)}"'] \arrow[u,"{(I', \lambda^{\prime \ast})}"']
\end{tikzcd}
\]
commutes.
\end{Def}
In this definition, the center $(\Pi_A)_0$ of $\Pi_A$ is necessarily the identity map on  the given vector bundle $A$.

 Given a Lie pair $(L,A)$, the $\v$-linear extension of $i$ defines a strict morphism $
 	I \colon A_\v^0 \to L^0_\v,
 $ between  the $\v$-Cartesian extensions of $L$ and $A$. This is indeed the \textit{trivial} infinitesimal deformation of $(L,A)$. Next, we  present a class of  `nontrivial' infinitesimal deformations.
Note that,   each Lie pair $(L,A)$ determines a short exact sequence of vector bundles over $M$:
\begin{equation}\label{Eq: SES}
	0 \to A \xrightarrow{i} L \xrightarrow{\prb} B:= L/A \to 0.
\end{equation}
Choose a splitting of the exact sequence above, given by an injective vector bundle map $j \colon B \to L$ and a surjective vector bundle map $\pra\colon L\to A$, which induces an isomorphism $L \cong A\oplus B$.
 For each element $\xi \in \Gamma(A^\ast \otimes B) \otimes \maximalidealofartin$, consider the following strict morphism of $\v$-ringed vector bundles
\begin{align}\label{Eqt: Def of Ixi}
  I_{\xi} \colon \Gamma(A_\v) &\to \Gamma(L_\v) \cong \Gamma(A_{\v}) \oplus \Gamma(B_{\v}), \notag \\
   a &\mapsto I_\xi(a) := i(a) +  j(\xi(a)),
\end{align}
which in fact determines a strict bundle map
\begin{align}\label{Eq: anchor}
  \rho_{A_\v}^\xi \colon \Gamma(A_\v) &\to \Gamma(T_\v), \notag \qquad \text{(see Example~\ref{Ex:tangent-bundle})}\\
                                               a &\mapsto \rho^\xi_{A_\v}(a) := \rho_{L_\v^0}(I_\xi(a)),
\end{align}
 and a bracket
 \begin{align}\label{Eq: bracket}
    [\cdot, \cdot]_{A_\v}^\xi \colon \Gamma(A_\v) \times \Gamma(A_\v) &\to \Gamma(A_\v) \notag  \\
   (a_1,a_2) &\mapsto  [a_1,a_2]_{A_\v}^\xi :=\prav \big([I_\xi(a_1), I_{\xi}(a_2)]_{L_\v^0} \big).
 \end{align}
\begin{lem}\label{Prop:equivalentSTD}
Suppose that the map $I_\xi$ satisfies
\begin{equation}\label{Eqt:standarddeformation}
	I_{\xi} \left(\prav \big([I_{\xi}(a_1), I_{\xi}(a_2)]_{L_\v^0} \big)\right) = [I_\xi(a_1), I_\xi(a_2)]_{L_\v^0},
\end{equation}
for all $a_1, a_2 \in \Gamma(A)$. Then the anchor map $\rho_{A_\v}^\xi$~\eqref{Eq: anchor} and the bracket $[-, -]_{A_\v}^\xi$~\eqref{Eq: bracket} together  define an $\v$-ringed Lie algebroid structure on $A_\v$. Moreover, the strict morphism $I_\xi$ in ~\eqref{Eqt: Def of Ixi} is an infinitesimal deformation of $(L, A)$.

\end{lem}
\begin{proof}
 We first check that $(A_\v,[-,-]_{A_\v}^\xi, \rho_{A_\v}^\xi )$ is an $\v$-ringed Lie algebroid. It suffices to verify the Leibniz rule and the Jacobi identity. In fact, for all $a_1,a_2 \in \Gamma(A)$ and $f \in C^\infty(M,\K)$, we have
\begin{align*}
  [a_1, fa_2]_{A_\v}^\xi &= \prav([I_{\xi}(a_1), I_\xi(fa_2)]_{L_\v^0}) = \prav([I_{\xi}(a_1), fI_\xi(a_2)]_{L_\v^0}) \\
  &= \prav(\rho_{L_\v^0}(I_{\xi}(a_1))(f) \cdot I_\xi(a_2) + f[I_{\xi}(a_1), I_\xi(a_2)]_{L_\v^0}) \\
  &= \rho_{A_\v}^\xi(a_1)f a_2 + f [a_1,a_2]^\xi_{A_\v}.
\end{align*}
Meanwhile, for all $a_1,a_2,a_3 \in \Gamma(A_\v)$, using Equation~\eqref{Eqt:standarddeformation}, we have
\begin{align*}
  [a_1, [a_2,a_3]_{A_\v}^\xi]_{A_\v}^\xi &= \prav[I_\xi(a_1), I_{\xi}(\prav([I_{\xi}(a_2), I_\xi(a_3)]_{L_\v^0}))]_{L_\v^0} \\
  &= \prav[I_\xi(a_1), [I_{\xi}(a_2), I_\xi(a_3)]_{L_\v^0}]_{L_\v^0}.
\end{align*}
Thus, the Jacobi identity for $[-,-]_{L_\v^0}$ implies the Jacobi identity for $[-,-]_{A_\v}^\xi$.
Now using Equation~\eqref{Eqt:standarddeformation} again, we see that $I_{\xi} \colon A_\v \to L_\v^0$ is an inclusion of $\v$-ringed Lie algebroids whose center is the given inclusion $i \colon A \to L$. This proves that $I_\xi$ is an infinitesimal deformation of $(L,A)$.
\end{proof}
We also need the converse fact of Lemma \ref{Prop:equivalentSTD} which is  easily seen.
\begin{lem}\label{Prop:equivalentSTD2}
If ${A_\v}$ is endowed with   an $\v$-ringed Lie algebroid structure  such that $(I_\xi,\id)\colon A_\v\to L_\v^0$ is a morphism of $\v$-ringed Lie algebroids, then
\begin{compactenum}
\item The $\v$-ringed Lie algebroid structure on $A_\v$ is the one determined by Equations~\eqref{Eq: anchor}  and ~\eqref{Eq: bracket};
\item The strict morphism $I_\xi$   is an infinitesimal deformation of $(L, A)$;
\item The map $I_\xi$ is subject to Equation \eqref{Eqt:standarddeformation}.
\end{compactenum}
\end{lem}

Infinitesimal deformations of the form $I_\xi$ (subject to Equation \eqref{Eqt:standarddeformation}) will be referred to as   \textbf{standard deformations}.

Indeed, any infinitesimal deformation of Lie pairs is strictly isomorphic to a standard one:
Given any infinitesimal deformation $(I, \lambda^\ast)$ of $(L, A)$,  we can use the chosen splitting $j \colon B \to L$ of the short exact sequence~\eqref{Eq: SES} and the associated decomposition $L= A \oplus B$ to express $I$ explicitly. When restricted onto $\Gamma(A)$, the inclusion $I$ decomposes as follows:
\begin{align}
I \colon\Gamma(A) &\to \Gamma(L_\v) = \Gamma(A_{\v}) \oplus \Gamma(B_{\v}),\nonumber \\ \label{Eqt:bigiotaiAiB}
a &\mapsto  I(a) = \bigiota_{A}(a)+ j(\bigiota_{B}(a)),
\end{align}
where   $\bigiota_B\colon\Gamma(A)\to \Gamma(B_\v)$ and  $\bigiota_{A}\colon\Gamma(A)\to\Gamma(A_\v)$  are both $\K$-linear.  Moreover, they satisfy the following conditions.

\begin{lem}
 \begin{compactenum}
 \item The pair $(\bigiota_A, \lambda^\ast)$ defines a small automorphism of the $\v$-ringed vector bundle $A_\v$;
 \item The map
 \[
  \bigiota_{B}\circ\bigiota_{A}^{-1}\colon\Gamma(A)\to \Gamma(B)\otimes \maximalidealofartin
 \]
 is $\CinfMK$-linear.
   \item The corresponding element $\xi:=\bigiota_{B}\circ\bigiota_{A}^{-1}$ in $\Gamma(A^\ast \otimes B)\otimes\maximalidealofartin$ satisfies Equation \eqref{Eqt:standarddeformation}, thus the associated map $I_\xi \colon A_\v^\xi \to L_\v^0$ defines a standard deformation of $(L,A)$.
\end{compactenum}
\end{lem}
\begin{proof}
$(1)$  Since $I \colon A_\v \to L_\v$ is an $\v$-ringed vector bundle morphism covering a small automorphism $\lambda^\ast$ of $M_\v$,  it follows that $I(fa)=\lambda^\ast(f)I(a)$ for all $a \in \Gamma(A_\v)$ and $f \in C^\infty(M_\v)$. Thus, we have
\begin{equation}\label{Eqt:temptemp1}
\bigiota_{A}(fa)=\lambda^\ast(f)\bigiota_{A}(a),
\end{equation}
and
\begin{equation}\label{Eqt:temptemp2}
\bigiota_{B}(fa)=\lambda^\ast(f)\bigiota_{B}(a).
\end{equation}	
The center of $\bigiota_{A}$  is the identity $\id$ of $A$. This (according to Proposition \ref{prop: Isom of vector bundles}) implies that $\bigiota_{A}$ is an automorphism of $A_\v$.
Thus, by Equation \eqref{Eqt:temptemp1}, $\bigiota_{A}$ covers the small automorphism  $\lambda^\ast$ of $M_\v$. This proves Statement (1).

$(2)$  Note that,   the inverse  $\bigiota_{A}^{-1}$ of $\bigiota_{A}$ is a small automorphism of $A_\v$ covering $(\lambda^{\ast})^{-1}$, i.e.,
\[
\bigiota_{A}^{-1}(fa)=(\lambda^{\ast})^{-1} (f)\bigiota_{A}^{-1}(a).
\]
Combining this equation with Equation \eqref{Eqt:temptemp2},  we see that $\bigiota_{B}\circ\bigiota_{A}^{-1}$ is $\CinfMK$-linear.

$(3)$ Obviously,  $\bigiota_{B}\circ\bigiota_{A}^{-1}$ determines an element $\xi \in \Gamma(A^\ast \otimes B)\otimes\maximalidealofartin$ satisfying
\[
 I_\xi(a) = a + j(\xi(a)) = a+ j(\bigiota_{B}(\bigiota_{A}^{-1}(a))) = I(\bigiota_{A}^{-1}(a)).
\]
Since $I =\bigiota_A+j\circ\bigiota_B$ is an infinitesimal deformation, we have
\begin{equation}\label{Eqt:iota}
I([a_1,a_2]_{A_\v}) = [I(a_1), I(a_2)]_{L_\v^0},
\end{equation}
for all $a_1,a_2\in\Gamma(A_\v)$. Thus,
	\begin{eqnarray*}
		&&(I_\xi\circ\pra)	\big([I_{\xi}(a_1), I_{\xi}(a_2)]_{L_\v^0} \big)\\
		&=& (I_\xi\circ\pra) \big([I (\bigiota_{A}^{-1}(a_1)), I (\bigiota_{A}^{-1}(a_2))]_{L_\v^0} \big) \qquad \text{by Equation~\eqref{Eqt:iota}} \\
		&=& (I_\xi \circ\bigiota_{A}) \big([\bigiota_{A}^{-1}(a_1),\bigiota_{A}^{-1}(a_2)]_{A_\v}\big)  \\
		&=& I\big([\bigiota_{A}^{-1}(a_1),\bigiota_{A}^{-1}(a_2)]_{A_\v}\big)  \qquad\qquad \text{by Equation \eqref{Eqt:iota}} \\
		&=&[I\circ\bigiota_{A}^{-1}(a_1), I\circ\bigiota_{A}^{-1}(a_2)]_{L_\v^0}   \\
		&=& [I_{\xi}(a_1), I_{\xi}(a_2)]_{L_\v^0}.
	\end{eqnarray*}
\end{proof}

In the sequel, $([-,-]_{A_\v}^\xi,\rho_{A_\v}^\xi; I_\xi, \id )$, induced from $(I,\lambda^\ast)$ and the splitting $j \colon B \to L$, will be called the  \textbf{standard realization} of $(I,\lambda^\ast)$.

\begin{prop}\label{Lem:standard}
The standard realization  $([-,-]_{A_\v}^\xi,\rho_{A_\v}^\xi; I_\xi, \id )$ is strictly isomorphic to $(I,\lambda^\ast)$.
\end{prop}
\begin{proof}
The goal is to establish the following commutative diagram of morphisms of $\v$-ringed Lie algebroids:
\[
 \begin{tikzcd}
 	L^0_\v & L^0_\v \arrow[l,"{\id }"'] \\
 	({A_\v }, [-,-]_{A_\v}^\xi, \rho_{A_\v}^\xi)\arrow[u,"{(I_\xi, \id)}"]     & (A_\v, \rho_{A_\v}, [-, -]_{A_\v}). \arrow[l,"{(\bigiota_A,\lambda^\ast)}"']\arrow[u,"{(I, \lambda^\ast)}"']
 \end{tikzcd}
 \]
This diagram naturally commutes by definitions of these arrows.
It suffices to show that $\bigiota_A$ is a morphism of $\v$-ringed Lie algebroids covering  $\lambda^\ast$.
In fact, since the map $I =\bigiota_{A}+j \circ \bigiota_{B}$ is a morphism of $\v$-ringed Lie algebroids, we have
\[
\bigiota_A([a_1,a_2]_{A_\v}) = \pra\big([I(a_1), I(a_2)]_{L_\v^0}) = [\bigiota_A(a_1),\bigiota_A(a_2)]_{A_\v}^{\xi},
\]
and
\[
\rho_{A_\v}^{\xi}(\bigiota_A(a)) =\rho_{L_\v^0}(\bigiota_A(a)+j(\xi(\bigiota_A(a)))) = \rho_{L_\v^0}(I(a)),
\]
for all $a,a_1,a_2\in\Gamma(A)$.
Thus, $\bigiota_A$ intertwines the relative $\v$-ringed Lie algebroid structures.
\end{proof}

\begin{Ex}
  Consider a codimension $k$ foliation $\mathcal{F}$ in an $n$-dimensional real smooth manifold $M$. Let $F \subset T_M$ be the involutive subbundle tangent to the leaves of $\mathcal{F}$. Then $(T_M, F)$ is a Lie pair with $B = T_M/F$ the normal bundle of $\mathcal{F}$.
  It follows that the set of infinitesimal deformations of $(T_M, F)$ parameterized the algebra $\R_2[t]$ of dual numbers coincides with that of infinitesimal deformations of the foliation $\mathcal{F}$~\cite{Heitsch}, the latter of which is defined by a smooth family of involutive distributions $\{F_t\}_{t \in \R}$ in $T_M$ such that  $F_0 = F$.

  More precisely, given an infinitesimal deformation $F_t$ of $F$, each Riemannian metric on $T_M$ determines a family of splittings $\pr_{F_t} \colon T_M \to F_t$ of the short exact sequences of vector bundles (over $M$)
  \[
   0 \to F_t \to T_M \xrightarrow{\pr_{B_t}} B_t = T_M/F_t \to 0.
  \]
 According to~\cite{Heitsch}, there exists an element $\sigma \in \Gamma(F^\ast \otimes B)$ defined for all $X \in \Gamma(F)$ by
 \[
  \sigma(X) = \pr_B\left(\frac{d}{dt}\mid_{t=0} \pr_{F_t}(X)\right).
 \]
 Let
 \[
  \xi = \sigma \cdot t \in \Gamma(F^\ast \otimes B) \otimes \mathfrak{m}_{\R_2[t]}.
 \]
 Then the associated map $I_\xi  \colon F_{\R_2[t]} \to (T_M)_{\R_2[t]}$ satisfies Equation~\eqref{Eqt:standarddeformation}. This condition  is indeed equivalent to the vanishing of the integrability tensor $\Lambda_t$ of the deformation $F_t$ (see~\cite{Heitsch}*{Proposition 2.10}).
Thus, $I_\xi$ defines a standard deformation of $(T_M,F)$.
\end{Ex}

\subsection{Weak deformation functors and their standard realizations}
It is too restrictive to classify infinitesimal deformations up to strict isomorphism.
We now introduce isomorphism classes of infinitesimal deformations which form a larger set in the sense that they are defined up to \emph{small automorphisms} of the Cartesian extension $L^0_\v$ of $L$.
\begin{Def}\label{Def:full-and-semi-iso}
Two infinitesimal deformations of a Lie pair $(L,A)$, \[
([-,-]_{A_\v} ,\rho_{A_\v}; I, \lambda^\ast)\quad \text{and}\quad ([-,-]^\prime_{A_\v} ,\rho^\prime_{A_\v}; I^\prime, \lambda^{\prime^\ast}),
\]
are said to be weak isomorphic (resp. semistrict isomorphic) if there exist
\begin{compactenum}
\item an element $(\exp(\Delta)$, $\exp(\sigma(\Delta)))$ in the small automorphism group $\saut (L_\v^0)$ (resp. small inner automorphism group $\operatorname{sIAut}(L_\v^0)$) of the $\v$-Cartesian extension $L_\v^0$ of $L$, and
\item  an $\v$-ringed Lie algebroid morphism
\[
(\Pi_A, \lambda^\ast_A):~(A_\v, [-,-]^\prime_{A_\v} ,\rho^\prime_{A_\v})\to  (A_\v, [-,-]_{A_\v} ,\rho_{A_\v}),
\]
\end{compactenum}
 such that the following diagram
\[
\begin{tikzcd}
 L^0_\v & L^0_\v \arrow[l, "{(\exp(\Delta), \exp(\sigma(\Delta)))}"'] \\
    (A_\v, \rho_{A_\v}, [-, -]_{A_\v}) \arrow[u,"{(I,\lambda^\ast)}"]     & (A_\v, \rho^\prime_{A_\v}, [-, -]^\prime_{A_\v}) , \arrow[l,"{(\Pi_A,\lambda_A^\ast)}"'] \arrow[u,"{(I^\prime, \lambda^{\prime\ast})}"']
\end{tikzcd}
\]
commutes.
\end{Def}

In this definition, the center $(\Pi_A)_0$ of $\Pi_A$ is  necessarily the identity map on  the given vector bundle $A$.

\begin{Def}\label{Def:infdefmfunctor}
The weak deformation functor ${\infDefFctLA}$, associated with the Lie pair $(L,A)$, is a functor  from the category $\ArtCat$ of local Artinian $\K$-algebras to the category \textbf{Set} of sets. Specifically, for each $\v\in \ArtCat$, the resulting set ${\infDefFctLA}(\v)$ is defined as:
\[
	{\infDefFctLA}(\v):=
	\Bigg\{\begin{aligned}
		~&\textrm{weak isomorphism classes of }\\
		&\textrm{ infinitesimal deformations of }\\
		&(L,A)~\textrm{ parameterized by   } \v
	\end{aligned}\Bigg\}.
\]
\end{Def}

For a morphism $\vartheta$ in the category	$\ArtCat$, the resulting map of sets ${\infDefFctLA}(\vartheta)$ is naturally defined.  This convention is adopted for all types of deformation functors throughout the paper.

Similarly, we denote by $\sinfDefFctLA$ the \textbf{semistrict deformation functor}, which sends each $\v$ to the set of semistrict isomorphism classes of $\v$-parameterized infinitesimal deformations of $(L,A)$.

Next, we   represent the weak and the  semistrict  infinitesimal deformation functors by  standard realizations. Let us fix a splitting $j$ of the short exact sequence~\eqref{Eq: SES} and denote by ${\rm Sd}(L,A,\v)$ the set of solutions $\xi \in \Gamma(A^* \otimes B) \otimes \maximalidealofartin$  to Equation~\eqref{Eqt:standarddeformation}, which can be identified  (via the chosen splitting $j$) with the set of standard deformations of $(L,A)$.

The  small automorphism group $\saut (L_\v^0)$ of the $\v$-Cartesian Lie algebroid $L_\v^0$ acts on the set of infinitesimal deformations, and   on the set ${\rm Sd}(L,A,\v)$ accordingly. Let us explain this fact.

\begin{compactenum}
  \item Given $\Pi\in \saut (L_\v^0)$ and $\xi\in \Gamma(A^* \otimes B) \otimes \maximalidealofartin$, we first establish the following commutative diagram in the category of $\v$-ringed vector bundles:
 \begin{equation}\label{Diag:temp1}
 \begin{tikzcd}
 	L^0_\v & L^0_\v \arrow[l,"{\Pi}"'] \\
 	{A_\v}\arrow[u,"{I_{\Pi \triangleright \xi}}"]     & {A_\v }. \arrow[l,"{ \Pi_\xi}"']\arrow[u,"{I_\xi}"']
 \end{tikzcd}
 \end{equation}
In this diagram,  we define
\[
 \Pi_\xi:=\pra\circ \Pi \circ I_\xi \colon \quad A_\v \to A_\v ,
 \]
as a morphism  of $\v$-ringed vector bundles. Note that,   the center of $\Pi_\xi$ is the identity map $\id \colon A \to A$. By Proposition~\ref{prop: Isom of vector bundles}, $\Pi_\xi$ is a small automorphism of the $\v$-ringed vector bundle $A_\v$. So we can consider the map
 \[
  \prb \circ \Pi \circ I_{\xi} \circ \Pi_\xi^{-1}: \quad A_\v \to B_\v.
\]
  It can be easily seen that it corresponds to an element   $\Pi \triangleright \xi\in \Gamma(A^\ast \otimes B) \otimes \maximalidealofartin $. Moreover, the morphism $I_{\Pi \triangleright \xi}=\id+j\circ{\Pi \triangleright \xi}  \colon A_\v \to A_\v\oplus B_\v=L_\v$ (of $\v$-ringed vector bundles) satisfies
 \begin{align}\label{Eq: CD}
 	I_{\Pi \triangleright \xi}(\Pi_\xi(a)) &= \Pi_\xi(a) + j(\Pi \triangleright \xi(\Pi_\xi(a)))\notag \\
 	& = \pra(\Pi (I_\xi(a))) + (j\circ \prb)(\Pi(I_{\xi}(a)))\notag \\
 	& = \Pi(I_\xi(a)),
 \end{align}
 for all $a \in \Gamma(A)$. Therefore, we see that Diagram \eqref{Diag:temp1} is indeed commutative.

  \item If $\xi$ is in the smaller subset ${\rm Sd}(L,A,\v)$($\subset \Gamma(A^* \otimes B) \otimes \maximalidealofartin$), then $I_\xi$ is an infinitesimal deformation of the Lie pair $(L,A)$. In this situation, we can equip $A_\v$ with a new $\v$-ringed Lie algebroid structure $([\cdot,\cdot]_{A_\v}', \rho_{A_\v}')$ by pulling back the original one $( [\cdot,\cdot]_{A_\v}^\xi, \rho_{A_\v}^\xi)$ through the small automorphism $\Pi_\xi$.  In doing so, we obtain the following commutative diagram in the category of $\v$-ringed Lie algebroids:
    \[
 \begin{tikzcd}
 	L^0_\v & L^0_\v \arrow[l,"{\Pi}"'] \\
 	({A_\v}, [\cdot,\cdot]'_{A_\v}, \rho'_{A_\v})\arrow[u,"{I_{\Pi \triangleright \xi}}"]     & ({A_\v }, [\cdot,\cdot]_{A_\v}^\xi, \rho_{A_\v}^\xi). \arrow[l,"{ \Pi_\xi}"']\arrow[u,"{I_\xi}"']
 \end{tikzcd}
 \]
  \item According to Lemma \ref{Prop:equivalentSTD2}, the $\v$-ringed  Lie algebroid $({A_\v}, [\cdot,\cdot]'_{A_\v}, \rho'_{A_\v})$ is exactly $({A_\v}$, $[\cdot,\cdot]_{A_\v}^{\Pi \triangleright \xi}$, $\rho_{A_\v}^{\Pi \triangleright \xi})$, and the element  $\Pi \triangleright \xi$  belongs to ${\rm Sd}(L,A,\v)$.
 \end{compactenum}

As a summary of the above  construction,  we have the desired action
 \begin{align}\label{Eq:Def of action}
   - \triangleright - \colon \saut (L_\v^0) \times {\rm Sd}(L,A,\v) &\to {\rm Sd}(L,A,\v), \notag \\
      (\Pi, \xi) &\mapsto \Pi \triangleright \xi := \prb \circ \Pi \circ I_{\xi} \circ \Pi_\xi^{-1}.
 \end{align}
 In a similar fashion, the group   $\operatorname{sIAut}(L_\v^0)$ of small inner automorphisms of    $L_\v^0$   also acts on the set ${\rm Sd}(L, A, \v)$. So the following   proposition is now obvious.

\begin{prop}\label{prop: isom standard deformation}
For any $\xi, \eta \in {\rm Sd}(L,A,\v)$, the standard deformations $I_\xi$ and $I_\eta$ are weak isomorphic (resp. semistrict isomorphic) if and only if $\xi$ and $\eta$ are in the same orbit, i.e., $\eta=\Pi \triangleright \xi$ for some $\Pi \in \saut (L_\v^0)$ (resp. $\operatorname{sIAut}(L_\v^0)$).

As a consequence, there exists a one-to-one correspondence between the set $\infDefFctLA(\v)$ of weak isomorphism classes of infinitesimal deformations (resp. $\sinfDefFctLA(\v)$ of semistrict isomorphism classes) and the orbit space ${\rm Sd}(L, A, \v) / \saut (L_\v^0)$ (resp. ${\rm Sd}(L, A, \v) / \operatorname{sIAut}(L_\v^0)$).
\end{prop}

Determining whether two elements, $\xi$ and $\eta$, in ${\rm Sd}(L, A, \v)$ belong to the same orbit under the action of either $\operatorname{sAut}(L_\v^0)$ or $\operatorname{sIAut}(L_\v^0)$ is a significant challenge. The primary obstacle lies in the need to explicitly compute the inverse $\Pi_\xi^{-1}$ within the group action map defined in Equation~\eqref{Eq:Def of action}, a task that is often analytically difficult. To circumvent this difficulty, we will introduce an alternative approach in the subsequent section.

\subsection{Gauge equivalence of Maurer-Cartan elements}

In deformation theory,    every formal deformation problem should be governed by an $L_\infty$-algebra. Specifically, isomorphic deformations correspond to gauge-equivalent Maurer-Cartan elements within this $L_\infty$-algebra. In this section, we investigate the weak (resp. semistrict) deformation functor of Lie pairs by means of the associated $L_\infty$-algebras. To begin, we recall some fundamental concepts related to $L_\infty$-algebras, as presented in~\cite{Getzler}. Note that our sign convention differs slightly from that used in~\cite{LM95}.

\begin{Def}
An $L_\infty$-algebra is a graded vector space $\g$ equipped with a collection of skew-symmetric maps $[\cdots]_k\colon \wedge^k \g\to \g$ of degree $(2-k)$ for all $k\geqslant 1$, called the $k$-bracket, satisfying the $n$-Jacobi identity
\[
\sum_{i=1}^n (-1)^i \sum_{\sigma \in {\rm Sh}(i,n-i) }\chi(\sigma) [[x_{\sigma(1)},\cdots, x_{\sigma(i)}]_i, x_{\sigma(i+1)},\cdots, x_{\sigma(n)}]_{n-i+1}=0,
\]
for all $n\geqslant 1$ and all homogeneous elements $x_1,\cdots, x_n \in \g$. Here ${\rm Sh} (i,n-i)$ is the set of $(i,n-i)$-shuffles, and $\chi(\sigma)$ is the Koszul sign of the $(i,n-i)$-shuffle $\sigma$ of the $n$-input $(x_1,\cdots, x_n)$.
\end{Def}
In particular, the $1$-bracket $[-]_1\colon \g\to \g$ is of degree $1$ and square zero, thus defines a cochain complex. We denote by $\H^i (\g, [-]_1)$ the $i$-th cohomology of $\g$ with respect to $[-]_1$.

An $L_\infty$-algebra with $[\cdots]_k=0$ for all $k\geqslant 3$ is a \textbf{dg Lie algebra}.
An $L_\infty$-algebra with $[\cdots]_k=0$ for all $k\geqslant 4$ is called a  \textbf{cubic $L_\infty$-algebra}~\cite{GMS} (also known as an $\Linftythree$-algebra).

The \textit{lower center filtration} $F^i \g$ on an $L_\infty$-algebra $\g$ is the decreasing filtration defined by $F^1 \g=\g$ and,  for $i\geqslant 2$, is defined inductively by
\[
F^i \g=\sum \limits_{i_1+\cdots+i_k=i}[F^{i_1}\g,\cdots,F^{i_k}\g]_k.
\]
An $L_\infty$-algebra  $\g$ is called \textbf{nilpotent} if the lower center series terminates,
that is,  $F^i \g=0$ for $i \gg 0$.
\begin{Ex}\label{Ex: tensor extension}
If $\g$ is an $L_\infty$-algebra  and $\v$ is a local Artinian $\K$-algebra with maximal ideal $\maximalidealofartin $, then the $\v$-extensions of all $k$-brackets on  $\g\otimes\maximalidealofartin$, i.e.,
\[
[x_1\otimes v_1, \cdots, x_k \otimes v_k]_k=[x_1, \cdots, x_k]_k\otimes v_1\cdots v_k,
\]
for all $x_i \in \g, v_i \in \maximalidealofartin$, together make $\g\otimes\maximalidealofartin$
into a nilpotent $L_\infty$-algebra.
\end{Ex}

 A \textbf{Maurer-Cartan element} of a nilpotent $L_\infty$-algebra $\g$ is a degree $1$ element $\xi \in \g^1$ satisfying the following Maurer-Cartan equation
	\[
		\sum_{k=1}^\infty \dfrac{1}{k!}[\xi,\cdots,\xi]_k=0.
	\]
Denote by $\MC(\g)$ the set of Maurer-Cartan elements of a nilpotent $L_\infty$-algebra $\g$.
Any Maurer-Cartan element $\xi \in \MC(\g)$ determines a new sequence of brackets
\begin{align*}
[x_1,\cdots, x_i]^\xi_i &=\sum_{k=0}^\infty \frac{1}{k!}[\xi^{\wedge k}, x_1, \cdots, x_i]_{k+i},
\end{align*}
where $[\xi^{\wedge k}, x_1,\cdots, x_i]_{k+i}$ is an abbreviation for
$[\xi, \cdots, \xi, x_1, \cdots, x_i]_{k+i}$, in which $\xi$ occurs $k$ times. We call $[\cdots]^\xi_i$ the $i$-th $\xi$-bracket.
These $\xi$-brackets $\{[ \cdots ]^\xi_i\}_{i \geqslant 1}$ defines a new nilpotent $L_\infty$-algebra structure on $\g$ (see~\cite{Getzler}*{Proposition 4.4}).

Two Maurer-Cartan elements $\xi, \eta \in \MC(\g)$ are said to be \textbf{gauge equivalent} if they are connected by \textbf{the $L_\infty$-exponential} of an element $b \in \g^0$ in the following sense:
\[
 \eta = e^b *\xi :=\xi-\sum_{k=1}^\infty \frac{1}{k!} e^k_\xi(b) \in \g^1,
\]
where $e^1_\xi(b)= [b]^\xi_1$, and the components $e^{k+1}_\xi(b) \in \g^1$ for $k \geqslant 1$ are inductively determined by
\[
 e^{k+1}_\xi(b) = \sum_{n=0}^\infty \frac{1}{n!} \sum_{\substack{k_1+\cdots+k_n=k \\ k_i\geqslant 1}} \frac{k!}{k_1!\cdots k_n!} [b,e^{k_1}_\xi(b),\cdots,e^{k_n}_\xi(b)]^\xi_{n+1}.
\]
\begin{Rem}
The definition of $L_\infty$-exponentials arises from Getzler's formula for the generalized Campbell-Hausdorff series, which is expressed as a sum of terms indexed by rooted trees. Specifically, Proposition 5.9 in Getzler's work provides a detailed exposition of this formula. Notably, when applied to a dg Lie algebra $\mathfrak{g}$, the $L_\infty$-exponential map reduces to the classical exponential map for dg Lie algebras.
 \end{Rem}

\begin{Def}\label{Def:deformationfunctor}
	The algebraic deformation functor (associated to an $L_\infty$-algebra $\g$)
	\[
	{\firstDef}\colon \ArtCat \to \textbf{Set},
	\]
	sends each local Artinian $\K$-algebra $\v \in \ArtCat$ to the set ${\firstDef}(\v)$ of gauge equivalent classes in $\MC(\g \otimes \maximalidealofartin)$, and each morphism $\vartheta\colon \v\to \v'$ to the map
	\begin{align*}
		\firstDef(\vartheta) \colon   {\firstDef}(\v)&\to   {\firstDef}(\v'), \\
		[\xi] &\mapsto   [(\id_\g\otimes \vartheta)\xi],
	\end{align*}
	for all $ \xi\in \MC(\g \otimes \maximalidealofartin )$.
\end{Def}

\begin{Ex}\label{Rem:tangentspace}
Consider the algebra $\v = \K_2[t] = \K[t]/(t^2)$ of dual numbers.
Note that,
\[
 \MC(\g \otimes \mathfrak{m}_{\K_2[t]}) \cong \{ \xi \otimes t | \xi \in \g^1,  [\xi]_1=0\},
\]
and that $b \in \g^0 \otimes \mathfrak{m}_{\K_2[t]}$ acts on $\xi$ by $e^b \ast \xi = \xi - [b]_1$. Therefore, the set $\firstDef(\K_2[t])$ is isomorphic to the first cohomology $\H^1(\g, [-]_1)$ of the $L_\infty$-algebra $\g$, thus a $\K$-vector space, known as \textbf{the tangent space} of the functor ${\firstDef}$.
\end{Ex}

We now describe the governing   $L_\infty$-algebra of infinitesimal deformations of $(L,A)$.
In fact, according to~\cite{BCSX}*{Proposition 4.1},  each splitting $j \colon B \to L$ of the short exact sequence~\eqref{Eq: SES} of vector bundles induces a cubic $L_\infty$-algebra structure $\{[\cdots]_k\}_{k=1}^3$ on the graded vector space $$\Omega^\bullet_A(B) := \Gamma(\wedge^\bullet A^*\otimes B) ,$$ whose unary bracket $[-]_1$ is the Chevalley-Eilenberg differential
\[
  \dCE \colon \Omega^\bullet_A(B) \to \Omega_A^{\bullet+1}(B)
\]
of the $A$-module structure on $B$  from the flat Bott $A$-connection $\nabla$ on $B$ defined by
\[
  \nabla_a b = \pr_B([a,j(b)]_L),\quad \forall a \in \Gamma(A), b \in \Gamma(B).
\]

For each local Artinian $\K$-algebra $\v$ with maximal ideal $\maximalidealofartin $, by Example~\ref{Ex: tensor extension}, the $\v$-linear extension $\Omega_A^\bullet(B) \otimes \maximalidealofartin$ of $\Omega_A^\bullet(B)$, when equipped with the $\v$-linear extension of the structure maps $\{[\cdots]_k\}_{k=1}^3$,  is a nilpotent cubic $\Linfty$-algebra.
By abuse of notations, we also denote these extended structure maps on $\Omega_A^\bullet(B) \otimes \maximalidealofartin$ by $\{[\cdots]_k\}_{k=1}^3$.

\begin{prop}\label{lem:MCstandard}
Given a splitting $j \colon  B \to L$ of the short exact sequence~\eqref{Eq: SES}, for each $\xi \in \Omega^1_A(B) \otimes \maximalidealofartin$, the map $I_{\xi} \colon A_\v \to L_\v$ defined in~\eqref{Eqt: Def of Ixi} is a standard deformation of $(L,A)$ if and only if $\xi\in \MC(\Omega_A^\bullet(B) \otimes \maximalidealofartin)$.
\end{prop}
\begin{proof}
For each $\xi \in \Omega_A^1(B) \otimes \maximalidealofartin$, according to the formulas in \cite{BCSX}*{Proposition 4.3}, the elements $[\xi]_1, [\xi,\xi]_2, [\xi,\xi,\xi]_3 \in \Omega_A^2(B) \otimes \maximalidealofartin$ are defined by for all $a_1, a_2 \in \Gamma(A)$,
\begin{align*}
  [\xi]_1(a_1, a_2) &= \prb([a_1, j(\xi(a_2))]_{L_\v^0})  - \prb([a_2, j(\xi(a_1))]) - \xi([a_1,a_2]_A), \\
  \frac{1}{2!}[\xi,\xi]_2(a_1,a_2) &= \prb([j(\xi(a_1)),j(\xi(a_2))]_{L_\v^0}) - \xi(\pra([j(\xi(a_1)),a_2]_{L_\v^0}))\\
  &\qquad - \xi(\pra([a_1, j(\xi(a_2))]_{L_\v^0})), \\
  \frac{1}{3!}[\xi,\xi,\xi]_3(a_1,a_2) &= - \xi(\pra[j(\xi(a_1)), j(\xi(a_2))]_{L_\v^0}).
\end{align*}
Thus, we have
\begin{align*}
	&\quad (\dCE \xi+\frac{1}{2}[\xi,\xi]_2+\frac{1}{6}[\xi,\xi,\xi]_3)(a_1,a_2)\\
  &=\prb([a_1, j(\xi(a_2))]_{L_\v^0})  - \prb([a_2, j(\xi(a_1))]_{L_\v^0})\\
  &\qquad - \xi(\pra[I_\xi(a_1), I_{\xi}(a_2)]_{L_\v^0}) +  \prb([j(\xi(a_1)),j(\xi(a_2))]_{L_\v^0}) \\
&= [I_\xi(a_1), I_\xi(a_2)]_{L_\v^0} - I_\xi(\pra[I_\xi(a_1), I_\xi(a_2)]_{L_\v^0}).
\end{align*}
By Lemma~\ref{Prop:equivalentSTD}, Equation \eqref{Eqt:standarddeformation} in particular, the inclusion $I_\xi$ defines a standard deformation of $(L,A)$ if and only if $\xi$ solves  the Maurer-Cartan equation of the cubic $L_\infty$-algebra $\Omega_A^{\bullet}(B) \otimes \maximalidealofartin$.
\end{proof}

It is natural to expect that the basic cubic $L_\infty$-algebra $\Omega_A^\bullet(B)$ controls the  infinitesimal deformations of $(L,A)$. That is, the associated algebraic deformation functor is isomorphic to the weak (or semistrict) deformation functor.
However, the degree $0$ component $\Gamma(B)$ of $\Omega_A^\bullet(B)$ cannot generate (by the  $L_\infty$ exponentials) the weak symmetry group ${\rm Aut}(L)$ of the Lie algebroid $L$ that acts on infinitesimal deformations of the Lie pair $(L,A)$.
To remedy this issue, we introduce another $L_\infty$-algebra $\h$, which is the extension of the basic cubic $L_\infty$-algebra $\Omega^\bullet_A(B)$ by the Lie algebra $\Der(L)$ of derivations of the Lie algebroid $L$.

\begin{prop}[\cite{NCCH}] \label{prop:extended}
The graded vector space
 \[
   \h := (\Der(L) \oplus \Gamma(B)) \bigoplus \left(\bigoplus_{n\geqslant 1}\Omega_A^n(B)\right)
 \]
is a cubic $L_\infty$-algebra extending the structure maps of the basic $L_\infty$-algebra $\Omega_A^\bullet(B)$ as follows:
\begin{compactenum}
  \item The $1$-bracket $[-]_1 \colon \Der(L) \to \Omega_A^1(B)$ is given by
  \[
      [\delta]_1 (a) = -\prb(\delta(a)),
  \]
  for all $ \delta \in \Der(L)$, $a \in \Gamma(A)$.
  \item The extended $2$-bracket $[-]_2$ is determined by the canonical commutator
   \[
      [\delta_1, \delta_2]_2  := \delta_1 \delta_2 - \delta_2 \delta_1,
   \]
  for all $\delta_1,\delta_2 \in \Der(L)$, and the Lie algebra $\Der(L)$  action on $\Omega_A^\bullet(B)$
  defined by
\begin{align*}
	[\delta, X]_2(a_1,\cdots,a_k) &:= (\prb \circ \delta \circ j)(X(a_1,\cdots,a_k)) - \sum_{j=1}^{k}X(a_1,\cdots,\pra\delta(a_j), \cdots,a_k),
\end{align*}
 for all $\delta \in \Der(L), X \in \Omega_A^k(B)$, and all $a_1,\cdots,a_k \in \Gamma(A)$.
  \item The $3$-bracket is the operation
  \[
    [-]_3 \colon \Der(L) \otimes \Omega^p_A(B) \otimes \Omega^q_A(B) \to \Omega_A^{p+q-1}(B)
  \]
  defined by
  \begin{align*}
    &\quad [\delta, X, Y]_3(a_1,\cdots,a_{p+q-1}) \\
    &= (-1)^{p+1}\sum_{\sigma \in \sh(p,q-1)}\sgn(\sigma) Y((\pra \circ \delta \circ j)X(a_{\sigma(1)},\cdots,a_{\sigma(p)}), a_{\sigma(p+1)},\cdots,a_{\sigma(p+q-1)}) \\
    &\quad + \sum_{\tau \in \sh(p-1,q)} \sgn(\tau) X((\pra \circ \delta \circ j) Y(a_{\tau(p+1)},\cdots, a_{\tau(p+q-1)} ),a_{\tau(1)},\cdots,a_{\tau(p)}),
  \end{align*}
  for all $\delta \in \Der(L), X\in \Omega_A^p(B), Y \in \Omega_A^q(B)$ and $a_1,\cdots,a_{i+j-1} \in \Gamma(A)$.
\end{compactenum}
\end{prop}
We call $\h$ the \textbf{extended cubic $L_\infty$-algebra} of $(L,A)$.
Denote by ${\rm Def}_{\h}$ the algebraic deformation functor associated to the $L_\infty$-algebra $\h$.

Our first main theorem of this paper states that the cubic $L_\infty$-algebra $\h$ controls weak deformations of the Lie pair $(L,A)$.
\begin{Thm}\label{Thm:MainTheorem}
The algebraic deformation functor  ${\rm Def}_{\h}$ associated to the   extended cubic $L_\infty$-algebra $\h$ of $(L,A)$   is naturally isomorphic to the
weak infinitesimal deformation functor ${\infDefFctLA}$.
\end{Thm}

In order to prove Theorem \ref{Thm:MainTheorem}, we need to build a natural transformation of functors
\begin{equation*}
\gamma\colon {\rm Def}_{\h}\Rightarrow {\infDefFctLA}.
\end{equation*}
Let us first define, for any local Artinian $\K$-algebra $\v$,  a map of sets
\[
  \gamma_\v \colon {\rm Def}_{\h}(\v) \to {\infDefFctLA}(\v).
\]
For every gauge equivalent class $[\xi] \in {\rm Def}_{\h}(\v)$, where $\xi \in \MC (\h \otimes\maximalidealofartin)$,  we define $\gamma_\v([\xi])$ to be the equivalence class of the standard deformation $([-,-]_{A_\v}^\xi,\rho_{A_\v}^\xi; I_\xi, \id )$, or the orbit in ${\rm Sd}(L, A, \v)/ \saut (L_\v^0)$ passing through $\xi$ by Proposition~\ref{prop: isom standard deformation}.
To see that $\gamma_\v$ is well-defined, we prove the following proposition.
\begin{prop}\label{Lem:gaugeEquivalent}
	Two Maurer-Cartan elements $\xi, \eta \in \MC(\h \otimes\maximalidealofartin)$ are gauge equivalent if and only if the associated standard deformations $I_\xi$ and $I_\eta$ of the Lie pair $(L,A)$ are weak isomorphic.
\end{prop}

We  note that the two Maurer-Cartan elements $\xi$ and $\eta$ are gauge equivalent means that  there exists a nilpotent derivation $\delta \in \Der(L) \otimes\maximalidealofartin$ of $L_\v^0$ such that
\[
\eta= e^{\delta} \ast \xi := -\sum_{k=0}^\infty \frac{1}{k!} e^k_\xi(\delta ),
\]
where $e^k_\xi(\delta ) \in \Omega_A^1(B)\otimes \maximalidealofartin, k\geqslant 0$ are inductively defined by $e^0_\xi(\delta)=-\xi$,
\begin{equation}\label{Eqt:e1xib0}
	e^1_\xi(\delta )=d_1(\delta) - [\delta,\xi]_2 +\frac{1}{2}[\delta,\xi,\xi]_3
\end{equation}
and
\begin{equation}\label{Eqt:enxib0}
	e^{k+1}_\xi(\delta )=[\delta, e^k_\xi(\delta)]_2 - [\delta, \xi,e^k_\xi(\delta)]_3 + \frac{1}{2}\sum_{\begin{subarray}{c}k_1+k_2=k\\k_1\geqslant 1,k_2\geqslant 1
	\end{subarray}} \frac{k!}{k_1! k_2!}[\delta, e^{k_1}_\xi(\delta ), e^{k_2}_\xi(\delta )]_3.
\end{equation}

Then we introduce two family of maps
\begin{align*}
x^k\colon &\Gamma(A_\v)\to \Gamma(A_\v),  & x^k(a) &:= (\prav \circ \delta^k \circ I_\xi) (a), \\
y^k\colon &\Gamma(A_\v)\to \Gamma(B_\v),  & y^k(a) &:= (\prbv  \circ\delta^k \circ I_\xi)(a),
\end{align*}
for all $k \geqslant 0$  and all $a \in \Gamma(A_\v)$. These maps satisfy the following key lemma, whose proof will be given in Appendix~\ref{SubsubSec:lemmaofgaugeequivalence}.
\begin{lem}\label{Lemma in appendix}
The  above maps $x^k$ and $y^k$ (for all $k \geqslant 0$) are related by the following relation
 \[
 	y^k = -\sum_{p=0}^k \tbinom{k}{p}e^p_\xi(\delta)\circ x^{k-p}, \quad \mbox{ as a map } \Gamma(A_\v)\to \Gamma(A_\v).
 \]
 \end{lem}
With the help of this lemma, we now prove Proposition \ref{Lem:gaugeEquivalent}.
\begin{proof}[Proof of Proposition \ref{Lem:gaugeEquivalent}] \emph{To see the necessity}, assume that two Maurer-Cartan elements $\xi$ and $\eta$ are gauge equivalent and let $\delta$, $x^k$, and $y^k$ be as earlier. Then we define a small automorphism of the $\v$-Cartesian extension $L^0_\v$ of the Lie algebroid $L$ by
\[
  \Pi:= \exp(\delta) = \sum_{k=0}^\infty \dfrac{\delta^k}{k!}\colon L^0_\v \to L^0_\v,
\]
and a small automorphism of the vector bundle $A_\v$ by
\[
 \Pi_A :=   \sum_{k=0}^\infty \frac{x^k}{k!} \colon A_\v \to A_\v .
\]
We claim that $(\Pi, \Pi_A)$ gives an isomorphism from $I_\xi$ to $ I_\eta$, i.e., the following commutative diagram
 \begin{equation}\label{Diag:PhiPhiA}
	\begin{tikzcd}
		L^0_\v  & L^0_\v \arrow[l,"{\Pi}"'] \\
		({A_\v}, [\cdot,\cdot]_{A_\v}^\eta, \rho_{A_\v}^\eta) \arrow[u,"{I_\eta}"]     & ({A_\v }, [\cdot,\cdot]_{A_\v}^\xi, \rho_{A_\v}^\xi) \arrow[l,"{\Pi_A}"]\arrow[u,"{I_\xi}"']
	\end{tikzcd}
\end{equation}
commutes in the category of $\v$-ringed Lie algebroids.
In fact, using Lemma~\ref{Lemma in appendix},  we have
\begin{align*}
\Pi \circ I_\xi &=\sum_{k=0}^\infty \frac{1}{k!}x^k + j \circ \sum_{k=0}^\infty \frac{1}{k!} y^k\\
&=\Pi_A- j \circ \sum_{k=0}^\infty \frac{1}{k!}\Big( \sum_{q=0}^k\tbinom{k}{q}e^q_\xi(\delta )\circ x^{k-q}\Big)  \\
&=\Pi_A-j\circ \Big(\sum_{p=0}^\infty \frac{1}{p!}e^p_\xi(\delta ) \Big)\circ
\Big( \sum_{k=0}^\infty \frac{1}{k!}x^k\Big)\\
&=\Pi_A + j \circ \eta \circ \Pi_A = I_\eta \circ\Pi_A,
\end{align*}
which implies that Diagram \eqref{Diag:PhiPhiA} commutes in the category of $\v$-ringed vector bundles.
Meanwhile, since we have
\begin{align*}
\Pi_A([a_1,a_2]_A^\xi) &=(\prav  \circ \Pi \circ I_\xi)([a_1,a_2]_A^\xi)\\
&=\prav  \Big([(\Pi\circ I_\xi)(a_1), (\Pi \circ I_\xi)(a_2)]_{L_\v^0}\Big)\\
&=\prav  \Big([(I_{\eta}\circ\Pi_A)(a_1), (I_\eta \circ\Pi_A)(a_2)]_{L_\v^0}\Big)\\
&=(\prav \circ I_\eta) \Big([\Pi_A(a_1),\Pi_A(a_2)]_A^\eta\Big)\\
&=[\Pi_A(a_1),\Pi_A(a_2)]_A^\eta,
\end{align*}
and
\[
(\rho_A^\eta\circ\Pi_A)(a)= (\rho_{L_\v^0}\circ I_{\eta} \circ \Pi_A)(a) = (\rho_{L_\v^0} \circ \Phi \circ I_{\xi})(a) = (\rho_{L_\v^0} \circ I_{\xi})(a) = \rho_A^\xi(a),
\]
for any $a_1,a_2, a \in \Gamma(A_\v)$, it follows that Diagram \eqref{Diag:PhiPhiA} indeed commutes in the category of $\v$-ringed Lie algebroids. Hence, $I_\xi$ and $I_\eta$ are isomorphic.

\emph{Conversely, to prove sufficiency}, assume that the standard deformations $I_\xi$ and $I_\eta$ are isomorphic.
By Proposition~\ref{prop: isom standard deformation}, there exists a small automorphism $\Pi \in \saut (L^0_\v)$ of the $\v$-Cartesian extension $L^0_\v$ of $L$ such that
\[
\eta = \Pi \triangleright \xi = \prbv\circ \Pi \circ I_{\xi} \circ\Pi_A^{-1},
\]
where $\Pi_A^{-1}$ is the inverse of $\Pi_A := \prav\circ\Pi\circ I_{\xi}$.

By Proposition~\ref{prop-aut}, we may assume that $\Pi = \exp(\delta)$ for some $\delta \in \Der(L) \otimes \maximalidealofartin$.
Using Lemma~\ref{Lemma in appendix}, we have
\begin{align*}
\eta\circ\Pi_A &= \prbv\circ \Pi \circ I_{\xi} = \prbv \circ \exp(\delta) \circ I_{\xi} = \sum_{k=0}^{\infty} \frac{1}{k!} \prbv\circ \delta^k \circ I_{\xi} =
\sum_{k=0}^\infty \frac{1}{k!}y^k \\
&=-\sum_{k=0}^\infty \frac{1}{k!}\Big(\sum_{q=0}^k\tbinom{k}{q}e^q_\xi(\delta )\circ x^{k-q}\Big) =-\Big( \sum_{p=0}^\infty \frac{1}{p!} e^p_\xi(\delta ) \Big)\circ\Big( \sum_{k=0}^\infty \frac{1}{k!}x^k \Big) \\
&= -\Big( \sum_{p=0}^\infty \frac{1}{p!} e^p_\xi(\delta ) \Big) \circ \Pi_A.
\end{align*}
Since $\Pi_A$ is a small automorphism of the vector bundle $A_\v$, we obtain
\[
\eta=-\sum_{p=0}^\infty \frac{1}{p!}e^p_\xi(\delta )=e^\delta \ast \xi,
\]
which implies that $\eta$ and $\xi$ are gauge equivalent.
\end{proof}

We are now in a position to prove the main theorem.
\begin{proof}[Proof of Theorem \ref{Thm:MainTheorem}]
By Proposition~\ref{Lem:gaugeEquivalent}, the assignment
\[
 \gamma_\v \colon   {\rm Def}_{\h}(\v) \to {\infDefFctLA}(\v)
\]
is well-defined.
Given any morphism $\vartheta \colon \v \to \v'$ of local Artinian $\K$-algebras,
it is easy to see that the following diagram
\[
\begin{tikzcd}
 {\rm Def}_{\h}(\v)  \ar{r}{\gamma_\v} \ar{d}{{\rm Def}_{\h}(\vartheta)} &   {\infDefFctLA}(\v) \ar{d}{\infDefFctLA(\vartheta)} \\
 {\rm Def}_{\h}(\v') \ar{r}{\gamma_{\v^\prime}}  & {\infDefFctLA}(\v')
\end{tikzcd}
\]
commutes. Therefore, $\gamma$ is indeed a natural transformation.
By Proposition~\ref{Lem:gaugeEquivalent}, the natural transformation $ \gamma_\v $ is   injective for any Artinian algebra $\v$. It follows from Proposition \ref{Lem:standard} that any infinitesimal deformation is realized by some  Maurer-Cartan element. So the natural isomorphism $\gamma_\v$ is also surjective.
\end{proof}

Note that,  the extension of the basic $L_\infty$-algebra $\Omega^\bullet_A(B)$ by the Lie algebra $\operatorname{IDer}(L)$ of inner derivations of $L$, denoted by
  \[
  \h_0 := (\operatorname{IDer}(L) \oplus \Gamma(B)) \bigoplus \left(\bigoplus_{n\geqslant 1}\Omega_A^n(B)\right)
  \]
  is an $L_\infty$-subalgebra of $\h$.
By a similar argument as in the proof of Theorem~\ref{Thm:MainTheorem}, we obtain the following
\begin{Thm}\label{Thm: semistrict}
The algebraic deformation functor ${\rm Def}_{\h_0}$ associated to the cubic $L_\infty$-algebra $\h_0$ is naturally isomorphic to the semistrict deformation functor $\operatorname{sDef}_{(L,A)}$ of the Lie pair $(L,A)$.
\end{Thm}

As a consequence, we have the following corollary:
\begin{Cor}\label{Cor:3identifications}
The tangent space ${\infDefFctLA}(\K_2[t])$ (resp. $\operatorname{sDef}_{(L,A)}(\K_2[t])$) is  isomorphic to the tangent cohomology $\H^1(\h, [-]_1)$ (resp.  $\H^1(\h_0, [-]_1)$).
\end{Cor}

\begin{Rem}
A natural relation between infinitesimal deformations of a Lie pair $(L,A)$
and infinitesimal deformations of the Lie algebroid $A$ studied in~\cite{Crainic2008} can be inferred.
In fact, there exists a morphism $\phi = \{\phi_1,\phi_2\}$ of $L_\infty$-algebras from the cubic $L_\infty$-algebra $\h$ (or its $L_\infty$-subalgebra $\h_0$) to the dg Lie algebra $C^\bullet_{\mathrm{def}}(A)$, called the \textit{deformation complex} of $A$, which controls the infinitesimal deformations of $A$.
In particular, for every Maurer-Cartan element $\xi \in \MC(\h \otimes \m_{\K_2[t]})$, this $L_\infty$-morphism $\phi$ sends a weak (or semistrict) isomorphism class of infinitesimal deformations of $(L,A)$ represented by $[\xi] \in \H^1(\h,[-]_1)$  (or $\H^1(\h_0, [-]_1)$) to an infinitesimal deformation of $A$ represented by $[\phi(\xi)] \in \H^{1}_{\rm def}(A)$.
\end{Rem}
\subsection{Examples}
\subsubsection{Lie algebra pairs}
Here we compare two types of deformations of a Lie algebra pair --- the one defined in the present paper and the one introduced by Crainic-Sch\"{a}tz-Struchiner in \cite{Crainic2014}.

Let $\mathfrak{l}$ be a Lie algebra and $\a \subset \mathfrak{l}$  a Lie subalgebra. Suppose that $\mathfrak{a}$ is of dimension  $k$. Denote by ${\rm Gr}_k(\mathfrak{l})$ the Grassmannian manifold of $k$-dimensional subspaces of $\mathfrak{l}$. Following ~\cite{Crainic2014}, a \textit{deformation of the Lie subalgebra} $\a$ inside $\mathfrak{l}$ is a smooth curve $\a_t \in C^\infty([0,1],{\rm Gr}_k(\mathfrak{l}))$ such that $\a_0=\a$ and $\a_t$ are Lie subalgebras of $\mathfrak{l}$ for all $t \in [0,1]$. Two deformations $\a_t$ and $\tilde{\a}_t$
of $\a$ are said to be \textit{isomorphic} if there exists a smooth curve $g_t, t \in [0,1]$ in the connected and simply connected Lie group $G$ integrating $\mathfrak{l}$, such that $g_0$ is the identity of $G$ and $\tilde{\a}_t = {\rm Ad}_{g_t} \a_t$.
It can be verified that the set of isomorphism classes of deformations of $\a$ is isomorphic to the tangent space $\operatorname{sDef}_{(\mathfrak{l},\a)}(\K_2[t])$ of the semistrict deformation functor of the Lie pair $(\mathfrak{l},\a)$, and also isomorphic to the first Chevalley-Eilenberg cohomology $\H^{1}_{\rm CE}(\a,\mathfrak{l}/\a)$ of the $\a$-module $\mathfrak{l}/\a$.

On the other hand, by Theorem \ref{Thm:MainTheorem},  weak isomorphism classes of infinitesimal deformations of the Lie pair $(\mathfrak{l},\a)$ is controlled by the cubic $L_\infty$-algebra $\h = \Der(\mathfrak{l})\oplus \Hom (\Lambda^{\bullet} \a,\mathfrak{l}/\a)$).
In particular, by Corollary~\ref{Cor:3identifications}, we have the following identifications \[
{\rm wDef}_{(\mathfrak{l},\a)}(\K_2[t]) \cong {\rm Def}_{\h}(\K_2[t])
\cong \H^1(\h,[-]_1).
\]
In general, the two  cohomology spaces $\H^{1}_{\rm CE}(\a,\mathfrak{l}/\a)$ and $\H^1(\h,[-]_1)$ are different.
For example, let $\mathfrak{l} = \mathfrak{b}(3,\K)$ be the $6$-dimensional Lie algebra consisting of  $3\times 3$ upper triangular matrices. Consider a $3$-dimensional Lie subalgebra $\a$ of $\mathfrak{l}$  generated by $e_{11}, e_{12}$ and $e_{13}$.  Here $e_{ij}$ represents the $3\times 3$ matrix with  $1$ in the $(i,j)$-entry and zeros elsewhere. We have
	\[
	[e_{11},e_{12}]=e_{12},\quad [e_{11},e_{13}]=e_{13},\quad [e_{12},e_{13}]=0.
	\] 	
By direct computations, one obtains $\H^1_{\CE}(\a,\mathfrak{l}/\a) \cong \K^3$ and
$\H^1(\h,[-]_1) \cong \K^2$.

However, if $\mathfrak{l}$ is semisimple, then all derivations of $\mathfrak{l}$ are inner. In this case, the two deformation functors $\operatorname{wDef}_{(\mathfrak{l},\a)}$ and $\operatorname{sDef}_{(\mathfrak{l},\a)}$ are isomorphic.

\subsubsection{Extensions of Lie algebroids}
We now consider a particular example from a construction of extensions of Lie algebroids in \cite{Mehta-Zambon}.

Let $(M, \pi)$ be a Poisson manifold. Then the cotangent bundle $T^* M$ is a Lie algebroid with the anchor $\pi^\sharp \colon T^*M \to TM$ and the Lie bracket $[-,-]_{T^\ast M}$ defined by
\[
[\alpha,\beta]_{T^* M} = \mathcal{L}_{\pi^\sharp(\alpha)} \beta-\mathcal{L}_{\pi^\sharp(\beta)} \alpha-d\pi(\alpha,\beta),
\]
for all $\alpha, \beta\in \Omega^1(M)$.
Given a Poisson vector field $V$ (i.e.  $\mathcal{L}_V \pi=0$), there is an extension of Lie algebroids on $T^*M \oplus (M \times \mathbb{R})$, where the anchor map is defined by $\rho_V(\alpha+f) = \pi^\sharp(\alpha) + fV$ for all $\alpha \in \Omega^\bullet (M)$ and $f \in C^\infty(M, \mathbb{R})$, and the Lie bracket is defined by
\[
[\alpha+f,\beta+g]_V :=[\alpha,\beta]_{T^* M} + \pi^\sharp(\alpha)g - \pi^\sharp(\beta)f + f\mathcal{L}_{V}\beta - g\mathcal{L}_{V}\alpha+fV(g)-gV(f),
\]
for all $\alpha,\beta \in \Omega^1(M)$ and $f, g\in C^\infty(M)$.
It follows that $(L = T^*M \oplus (M\times \mathbb{R}), A=T^*M)$ is a Lie pair with $B = L/A \cong M \times \mathbb{R}$.

Suppose that $V$ is the Hamiltonian vector field of a smooth function $\phi \in C^\infty(M)$, i.e.   $V=\pi^\sharp(d\phi)$. Let us analyze the tangent space  ${\infDefFctLA}(\K_2[t]) = \H^1(\h,  [-]_1) $ of the weak deformation functor ${\infDefFctLA}$ of this Lie pair.
First of all, note that the set $\operatorname{MC}(\h)$ of Maurer-Cartan element of the associated cubic $L_\infty$-algebra $\h$ is indeed the set of Poisson vector fields on $M$.
In fact, any Poisson vector field $Y$ induces a derivation $\delta^Y\in  \Der (T^*M\oplus (M\times \R))$ with zero symbol defined by
\[
\delta^Y (\alpha+f):= -\langle Y,\alpha\rangle d\phi+\langle Y,\alpha\rangle-f   Y(\phi)d\phi,
\]
for all $\alpha + f \in \Gamma(T^\ast M \oplus (M \times \R))$.
One can directly examine that $[\delta^Y]_1=-Y$, which implies that $\H^1(\h,  [-]_1) = 0$.
Hence, every infinitesimal deformation (parameterized by $\K_2[t]$) of this particular Lie pair $(T^*M\oplus (M\times \mathbb{R}),T^*M)$ arising from a Hamiltonian vector field  is \textit{trivial}.

\section{Infinitesimal deformations of matched Lie pairs}\label{Sec:thirdpartmatchedpair}
Let $(L,A)$ be a matched Lie pair, i.e., the short exact sequence~\eqref{Eq: SES} admits a canonical splitting such that $B$ is also Lie subalgebroid of $L$. We denote such a matched Lie pair by $A \bowtie B$.
In this section, we study infinitesimal deformations of this particular type of Lie pairs.
\subsection{The deformation functor}Consider the subgroup $\operatorname{hAut}(L_\v^0)$ of the group $\saut (L_\v^0)$ of small automorphisms of $L_\v^0$ defined by
\[
  \operatorname{hAut}(L_\v^0)  := \{\exp(\operatorname{L}_b) \mid b \in \Gamma(B \otimes \maximalidealofartin)\}.
\]
We elect the notation ``$\operatorname{hAut}$'' because it arises from `half' of  the collection of inner derivations $\operatorname{L}_l:=[l,-]_L$ for $l\in \Gamma(L)\otimes \maximalidealofartin=\Gamma(A\oplus B)\otimes \maximalidealofartin$.
\begin{Def}\label{Def:half-iso}
Two infinitesimal deformations of a matched Lie pair $L = A \bowtie B$
\[
([-,-]_{A_\v} ,\rho_{A_\v}; I, \lambda^\ast)\quad \text{and}\quad ([-,-]^\prime_{A_\v} ,\rho^\prime_{A_\v}; I^\prime, \lambda^{\prime\ast})
\]
are said to be isomorphic if there exists an element $\exp(\operatorname{L}_b) \in \operatorname{hAut}(L_\v^0)$ and an $\v$-ringed Lie algebroid morphism $(\Pi_A,\lambda_A^\ast)$ from $(A_\v, [-,-]^\prime_{A_\v} ,\rho^\prime_{A_\v})$ to $(A_\v, [-,-]_{A_\v} ,\rho_{A_\v})$ whose center is the identity of $A$ such that the following diagram
\[
\begin{tikzcd}
 L^0_\v & L^0_\v \arrow[l, "{\exp(\operatorname{L}_b)}"'] \\
    (A_\v, \rho_{A_\v}, [-, -]_{A_\v}) \arrow[u,"{(I,\lambda^\ast)}"]     & (A_\v, \rho^\prime_{A_\v}, [-, -]^\prime_{A_\v}) \arrow[l,"{(\Pi_A,\lambda_A^\ast)}"'] \arrow[u,"{(I',\lambda^{\prime\ast})}"']
\end{tikzcd}
\]
commutes.

The assignment for each local Artinian $\K$-algebra $\v$  the set  $\operatorname{hDef}_{A \bowtie B}(\v)$  of isomorphism classes of infinitesimal deformations of the matched Lie pair $L = A \bowtie B$ parameterized by $\v$ determines a functor
\[
  \operatorname{hDef}_{A \bowtie B} \colon  \ArtCat \to \textbf{Set},
\]
called the deformation functor of the matched Lie pair $L = A \bowtie B$.
\end{Def}

Note that,  the basic $L_\infty$-algebra $\Omega_A^\bullet(B)$ of the matched Lie pair $L = A \bowtie B$ degenerates to a canonical dg Lie algebra, since its third bracket $[-,-,-]_3$ vanishes in this case.
 Here is the main theorem.
\begin{Thm}\label{Thm: matched pair}
The deformation functor $ \operatorname{hDef}_{A \bowtie B} $ of the matched Lie pair $L = A \bowtie B$ is isomorphic to the  algebraic deformation functor associated to the dg Lie algebra $\Omega_A^\bullet(B)$.
\end{Thm}

\begin{proof}
  Since $B \subset L$ is also a Lie subalgebroid in this case, by Proposition~\ref{Lem:standard}, each infinitesimal deformation $I$ admits a canonical standard realization $I_\xi$ for some $\xi \in \Gamma(A \otimes B^\ast) \otimes \maximalidealofartin$ satisfying~\eqref{Eqt:standarddeformation}.
On the other hand, by Proposition~\ref{prop: isom standard deformation}, the restriction of the map~\eqref{Eq:Def of action}
defines an action of $\operatorname{hAut} (L_\v^0)$ on the set ${\rm Sd}(L, A, \v)$ of standard realizations, such that the standard deformations $I_\xi$ and $I_\eta$ are  isomorphic if and only if $\eta=\Pi \triangleright \xi$ for some $\Pi \in \operatorname{hAut} (L_\v^0)$.
As a consequence, the set $\infDefFctLA(\v)$ of weak isomorphism classes of infinitesimal deformations is isomorphic to the set of orbits ${\rm Sd}(L, A, \v)/ \operatorname{hAut} (L_\v^0)$ of this subgroup action.

On the other hand, by Proposition~\ref{lem:MCstandard}, the set ${\rm Sd}(L, A, \v)$  of standard deformations can be identified with the set $\operatorname{MC}(\Omega_A^\bullet(B) \otimes \maximalidealofartin)$ of Maurer-Cartan elements of the dg Lie algebra $\Omega_A^\bullet(B) \otimes \maximalidealofartin$.
Meanwhile, by Proposition~\ref{Lem:gaugeEquivalent},  two Maurer-Cartan elements $\xi$ and $\eta$ are gauge equivalent via an inner derivation $\adLb$ of $L_\v^0$ for some $b\in\Gamma(B)\otimes\maximalidealofartin$,  if and only if standard deformations  $I_\xi$ and $I_\eta$ are isomorphic via $\exp(\adLb)\in\operatorname{hAut}(L_\v^0)$. The conclusion is thus immediate.
\end{proof}

\subsection{Examples}
\subsubsection{Deformation of complex structures}
Recall from \cite{Manetti} that an infinitesimal deformation of a complex manifold $X$ over ${\rm Spec}(\v)$ is a morphism of sheaves of $\C$-algebras $\mathcal{F}\to \mathcal{O}_X$ over $X$. It is required that $\mathcal{F}$  is flat over $\v$ and $\mathcal{F}\otimes_\v \C\to \mathcal{O}_X $ is
an isomorphism.  Two infinitesimal deformations $\mathcal{F}_1$ and $\mathcal{F}_2$ are said to be isomorphic if there exists an $\v$-linear isomorphism $\phi\colon \mathcal{F}_1 \to \mathcal{F}_2$ such that the following diagram is commutative
\[
\begin{tikzcd}
	\mathcal{F}_1  \ar{dr}\ar{rr}{\phi} &  & \mathcal{F}_2 \ar{dl} \\
	  & \mathcal{O}_X. &
\end{tikzcd}
\]

The assignment for each local Artinian $\C$-algebra $\v$ the set ${\rm Def}_X(\v)$ of isomorphism classes of infinitesimal deformations of $X$ over ${\rm Spec}(\v)$ defines a functor
\[
 {\rm Def}_X \colon \textbf{Art} \to \textbf{Set}.
\]
Note that,  the complexified tangent bundle $T^\mathbb{C}_X = T_X^{0,1} \bowtie T_X^{1,0}$ is a matched Lie pair.
We claim that the deformation functor $ {\rm Def}_X$ of complex structures on $X$ is isomorphic to the deformation functor $\operatorname{hDef}_{T_X^{0,1} \bowtie T_X^{1,0}}$ of the matched Lie pair $(T^\mathbb{C}_X, T_X^{0,1})$.

In fact, the dg Lie algebra $\Omega_A^\bullet(B)$ under this circumstance is exactly the Kodaira-Spencer algebra
\[
{\rm KS}_X := \Omega_X^{0,\bullet}(T_X^{1,0})  = \Gamma(\Lambda^\bullet (T^{0,1}_X)^*\otimes T^{1,0}_X).
\]
Any infinitesimal deformation of the matched Lie pair $(T^\mathbb{C}_X, T_X^{0,1})$ is realized by some  Maurer-Cartan element  $\xi \in \Omega_X^{0,1}(T_X^{1,0})\otimes_\C\maximalidealofartin$ such that it is of the form:
\[
I_\xi = i + \xi \colon T_X^{0,1} \otimes \v \hookrightarrow T^{\mathbb{C}}_X \otimes \v.
\]
Consider the composition
\[
I_\xi^\ast \circ d \colon \Omega_X^{0,0} \otimes \v \to \Omega_X^{0,1}\otimes \v
\]
of the de Rham differential $d$ and the linear dual $I_\xi^\ast$ of the bundle map $I_\xi$.  Here $\Omega_X^{0,0}$ coincides with the space of smooth $\mathbb{C}$-valued functions on $X$. It can be seen that    $I_\xi^\ast \circ d = \bar{\partial} + \xi\lrcorner\partial$, and hence we define a sheaf as the kernel
\[
\mathcal{F}_\xi := \ker(\Omega_X^{0,0} \otimes_\C \v \xrightarrow{\bar{\partial} + \xi\lrcorner\partial} \Omega_X^{0,1} \otimes \v).
\]
It is clear that $\mathcal{F}_\xi$, together with the evaluation map $\ev \colon \v \to \C$, defines an infinitesimal deformation $\mathcal{F}_\xi \to \O_X$ of $X$.

On the other hand, if two Maurer-Cartan elements $\xi$ and $\eta$ are gauge equivalent via  some $b \in \Gamma(T_X^{1,0}) \otimes \maximalidealofartin$, i.e., $\eta = e^b \ast \xi$.
Then
by Proposition~\ref{Lem:gaugeEquivalent}, the standard deformations $I_\xi$ and $I_\eta$ are isomorphic via the small automorphism $\exp(\adLb)$, i.e., $I_\eta = \exp(\adLb)\circ I_\xi$. Here $\adLb  \colon \Omega_X^{0,\bullet}\otimes \v \to \Omega_X^{0,\bullet}\otimes \v$ is the   inner Lie derivative     of $T_X^\C \otimes \v$ along $b$.
 It follows that the following diagram is commutative:
\[
\begin{tikzcd}
		 \Omega_X^{0,0}\otimes \v \arrow[d,"{e^{b}}"] \arrow[r,"I_\xi^\ast \circ d"] & \Omega_X^{0,1}\otimes \v \ar[d, "e^b"] \\
		 \Omega_X^{0,0}\otimes \v \arrow[r,"I_\eta^\ast \circ d"] & \Omega_X^{0,1}\otimes \v .
	\end{tikzcd}
\]

Further, it can be verified that $\mathcal{F}_\xi$ and $\mathcal{F}_\eta$ are isomorphic via the following diagram:
 \[
\begin{tikzcd}
		0 \arrow[r] & \mathcal{F}_\xi \arrow[r]\arrow[d,"{\cong}"] & \Omega_X^{0,0}\otimes \v\arrow[d,"{e^{b}}"] \arrow[r,"{\bar{\partial}+\xi\lrcorner\partial}"] & \Omega_X^{0,1} \otimes \v \ar[d, "e^b"] \\
		0 \arrow[r] & \mathcal{F}_\eta \arrow[r] & \Omega_X^{0,0}\otimes \v\ \arrow[r,"{\bar{\partial} + \eta\lrcorner\partial}"] & \Omega_X^{0,1} \otimes \v.
	\end{tikzcd}
\]

So  in this setting, what Theorem \ref{Thm: matched pair} states is exactly the well-known fact of
isomorphism of functors
  \[
     \gamma \colon \operatorname{Def}_{{\rm KS}_X} \cong \operatorname{Def}_X.
  \]
The reader may wish to see \cites{GM, Manetti} for more in-depth    discussions on this subject.

\subsubsection{Transversely holomorphic foliations}
Assume that $\mathcal{F}$ is a transversely holomorphic foliation defined on a compact manifold $M$, with real dimension $p$ and complex codimension $q$. Such a foliation structure leads to a connection between the geometry of the said manifold and the algebraic properties of the associated bundles \cite{DK}.

Let $F$ denote the tangent bundle   of the foliation $\mathcal{F}$. A key feature of a transversely holomorphic foliation is that its transverse complex structure induces a well-defined complex structure on the normal bundle $B := T_M/F$. This normal bundle captures the behavior of the manifold $M$ in directions transverse to the leaves of $\mathcal{F}$.
Upon complexifying the normal bundle, we obtain a decomposition
\[
B^{\mathbb{C}} = B^{1,0} \oplus B^{0,1},
\]
where $B^{1,0}$ and $B^{0,1}$ are the holomorphic and anti-holomorphic components respectively.

Let $F^{\mathbb{C}}$ be the complexified tangent bundle of the foliation $\mathcal{F}$. Then the direct sum
\[
A := F^{\mathbb{C}} \oplus B^{0,1}
\]
is a Lie subalgebroid of the complexified tangent bundle $T_M^{\mathbb{C}}$ of $M$. Thus, $(T_M^{\mathbb{C}}, A)$ is a Lie pair. Moreover, the quotient bundle $T_M^{\mathbb{C}}/A$ is naturally identified with $B^{1,0}$. Hence, we obtain a matched Lie pair
\[
T_M^{\mathbb{C}} = A \bowtie B^{1,0},
\]
which plays a crucial role in the deformation theory of the initial foliation $\mathcal{F}$.
In fact,  the set of isomorphism classes of infinitesimal deformations of this matched Lie pair $A \bowtie B^{1,0}$ is isomorphic to the set of infinitesimal deformations of the transversely holomorphic foliation $\mathcal{F}$~\cites{Spencer1,Spencer2, DK,Gomez}.
This correspondence provides a useful tool for studying the deformation theory of $\mathcal{F}$ from the perspective of Lie algebroids and their deformations. Thus, by Theorem~\ref{Thm: matched pair}, we recover the following fundamental result:
\begin{Thm}[\cites{DK,Gomez}]
	The infinitesimal deformations of the transversely holomorphic foliation $\mathcal{F}$ are controlled by the dg Lie algebra $\Gamma\big(\wedge^\bullet (F^{\mathbb{C}} \oplus B^{0,1})^\ast \otimes B^{1,0}\big)$.
\end{Thm}

To the best of our knowledge, comprehensive results on the existence of local moduli spaces for smooth foliations are currently lacking (see \cite{MNR}). Consequently, this problem warrants further investigation and is a direction for future research.

\appendix

\section{Proof of Lemma~\ref{Lemma in appendix}} \label{SubsubSec:lemmaofgaugeequivalence}

We proceed by induction on $k$. When $k=0$, since $x^0(a) = a, y^0(a) = \xi(a)$, Lemma~\ref{Lemma in appendix} holds trivially, as
\[
y^0(a)= \xi(a) = -e^0_\xi(\delta)(a) = -e^0_\xi(\delta) \circ x^0(a),
\]
for all $a \in \Gamma(A)$.

Suppose that Lemma~\ref{Lemma in appendix} holds for $k \leqslant n$ for some $n\geqslant 1$.
For the inductive step  $k=n+1$, we note that
\begin{align}\nonumber
x^{n+1}(a) &=(\prav \circ \delta^{n+1} \circ I_\xi) (a) = \prav \circ \delta \circ \bigl(\delta^{n}  (I_\xi(a)) \bigr)\\
	&=\prav \circ \delta \bigl(x^n(a)+j(y^n(a)) \bigr), \label{Eqn:inductivexk+10}
\end{align}
and that
\begin{equation}\label{Eqt:inductiveyk+10}
	y^{n+1}(a) = (\prbv \circ\delta)(x^n(a)+ (j \circ y^n)(a)).
\end{equation}
for all $a \in \Gamma(A)$.
We also need to compute $e^{n+1-m}_\xi(\delta) \circ x^m$ for all $0\leqslant m \leqslant n$. We divide the computation into the following three cases:
\begin{enumerate}
	\item For $m=0$, by Equation \eqref{Eqt:enxib0}, one has
\begin{align*}
	&e^{n+1}_\xi (\delta) = [\delta, e^n_\xi(\delta)]_2 - [\delta, \xi,e^n_\xi(\delta)]_3 + \frac{1}{2}\sum_{\begin{subarray}{c}p+q=n\\p\geqslant 1, q\geqslant 1
		\end{subarray}} \frac{n!}{p! q!}[\delta, e^{p}_\xi(\delta ), e^{q}_\xi(\delta )]_3 \\
	&=(\prbv \circ\delta\circ j) \circ e^n_\xi(\delta) - e^n_\xi(\delta) \circ\prav \circ\delta - e^n_\xi(\delta)\circ (\prav \circ\delta \circ j) \circ \xi - \xi \circ (\prav \circ\delta \circ j) \circ e^n_\xi(\delta)\\
	&\quad +\sum_{p=1}^{n-1}\tbinom{n}{p}e^{n-p}_\xi(\delta)\circ (\prav \circ\delta \circ j) \circ e^p_\xi(\delta)\\
	&=(\prbv \circ\delta\circ j) \circ  e^n_\xi(\delta) - e^n_\xi(\delta)\circ\prav \circ\delta \circ (x^0+j\circ y^0) +\sum_{p=1}^n\tbinom{n}{p}e^{n-p}_\xi(\delta)\circ\prav \circ\delta \circ e^p_\xi(\delta)\\
	&=(\prbv \circ\delta \circ j)\circ e^n_\xi(\delta) -e^n_\xi(\delta)\circ x^1 +
	\sum_{p=1}^n\tbinom{n}{p}e^{n-p}_\xi(\delta)\circ (\prav \circ\delta\circ j) \circ e^p_\xi(\delta).
	\end{align*}
\item  For all $m= 1,\cdots,n-1$, by Equation \eqref{Eqt:enxib0} and the inductive assumption,  we have
\begin{align*}
	e^{n-m+1}_\xi(\delta)& \circ x^m  =(\prbv \circ\delta\circ j) \circ  e^{n-m}_\xi(\delta)\circ x^m - e^{n-m}_\xi(\delta)\circ (\prav \circ\delta) \circ x^m \\
&\quad + e^{n-m}_\xi(\delta)\circ(\prav \circ\delta \circ j) \circ e^0_\xi(\delta)\circ x^m \\
&\qquad + \sum_{p=1}^{n-m}\tbinom{n-m}{p}e^{n-m-p}_\xi(\delta)\circ (\prav \circ\delta \circ j) \circ e^p_\xi(\delta)\circ x^m\\
	&=(\prbv \circ\delta \circ j)\circ e^{n-m}_\xi(\delta)\circ x^m - e^{n-m}_\xi(\delta)\circ\prav \circ \delta \circ (x^m+ j \circ y^m)\\
	&\qquad -e^{n-m}_\xi(\delta)\circ (\prav \circ\delta \circ j) \circ
	\Big( \sum_{i=0}^m\tbinom{m}{i}e^i_\xi(\delta)\circ x^{m-i} \Big) \\
&\quad\qquad +\sum_{p=0}^{n-m} \tbinom{n-m}{p}e^{n-m-p}_\xi(\delta)\circ (\prav \circ\delta \circ j) \circ e^p_\xi(\delta) \circ x^m\\
	&=(\prbv \circ\delta \circ j) \circ e^{n-m}_\xi(\delta)\circ x^m - e^{n-m}_\xi(\delta)\circ x^{m+1} \\
&\qquad - e^{n-m}_\xi(\delta)\circ (\prav \circ\delta \circ j) \circ
	\Big(\sum_{i=0}^m\tbinom{m}{i}e^i_\xi(\delta)\circ x^{m-i} \Big) \\
&\qquad +\sum_{p=0}^{n-m} \tbinom{n-m}{p}e^{n-m-p}_\xi(\delta) \circ (\prav \circ\delta \circ j) \circ  e^p_\xi(\delta)\circ x^m,
	\end{align*}
where we have used Equation \eqref{Eqn:inductivexk+10} in the last step.
\item For $m=n$, by Equation \eqref{Eqt:e1xib0}, we have
\begin{align*}
	&\quad e^1_\xi(\delta)\circ x^n \\
	&=-\prbv \circ\delta\circ x^n- (\prbv \circ\delta\circ j) \circ  \xi\circ x^n + \xi\circ (\prav \circ\delta) \circ x^n +\xi \circ(\prav \circ\delta\circ j) \circ \xi\circ x^n\\
	&=-\prbv \circ\delta \circ x^n- (\prbv \circ\delta \circ j)\circ \xi \circ x^n - e^0_\xi(\delta)\circ\prav \circ \delta \circ(x^n+j \circ y^n)\\
	&\quad -e^0_\xi(\delta)\circ (\prav \circ\delta\circ j) \circ \Big(\sum_{i=0}^n\tbinom{n}{i}e^i_\xi(\delta)\circ x^{n-i} \Big)\\
	&=-\prbv \circ\delta \circ x^n- (\prbv \circ\delta\circ j) \circ \xi\circ x^n-e^0_\xi(\delta)\circ x^{n+1}\\	
&\quad -e^0_\xi(\delta)\circ (\prav \circ\delta \circ j) \circ \Big(\sum_{i=0}^n\tbinom{n}{i} e^i_\xi(\delta)\circ x^{n-i} \Big).
	\end{align*}
\end{enumerate}
	Summing up the above three equalities, we obtain
	\begin{align*}
	\sum_{m=0}^n & \tbinom{n}{m}e^{n+1-m}_\xi(\delta)\circ x^m
	=e^{n+1}_\xi(\delta) + \sum_{m=1}^{n-1}\tbinom{n}{m}e^{n+1-m}_\xi(\delta)\circ x^m+
	e^1_\xi(\delta)\circ x^n\\
	&=(\prbv \circ\delta\circ j) \circ \Big(\sum_{m=0}^{n}\tbinom{n}{m}e^{n-m}_\xi(\delta)\circ x^m\Big) -\sum_{m=0}^{n}\tbinom{n}{m}e^{n-m}_\xi(\delta)\circ x^{m+1} - \prbv \circ\delta\circ x^n\\
	&\qquad -\sum_{m=1}^n \sum_{i=1}^m\tbinom{n}{m}\tbinom{m}{i}
		e^{n-m}_\xi(\delta)\circ(\prav \circ\delta \circ j) \circ e^i_\xi(b)\circ x^{m-i}\\
	&\qquad\qquad +\sum_{m=0}^{n-1}\sum_{p=0}^{n-m}\tbinom{n}{m}\tbinom{n-m}{p}
		e^{n-m-p}_\xi(\delta)\circ (\prav \circ\delta \circ j) \circ e^p_\xi(\delta)\circ x^m\\
	&=-\prbv \circ\delta\circ j \circ  y^{n} -\sum_{m=0}^{n}\tbinom{n}{m}e^{n-m}_\xi(\delta)\circ x^{m+1}-\prbv \circ\delta\circ x^n \quad \text{by Equation ~\eqref{Eqt:inductiveyk+10}} \\
	&=-y^{n+1}-\sum_{m=0}^{n}\tbinom{n}{m}e^{n-m}_\xi(\delta)\circ x^{m+1}.
	\end{align*}
Thus, we have
\begin{align*}
	y^{n+1}&=-\sum_{m=0}^n\tbinom{n}{m}e^{n+1-m}_\xi(\delta)\circ x^{m}
	-\sum_{m=0}^n\tbinom{n}{m}e^{n-m}_\xi(\delta)\circ x^{m+1}\\
	&=-\sum_{m=0}^n\tbinom{n}{m}e^{n+1-m}_\xi(\delta)\circ x^{m}
	-\sum_{m=1}^{n+1}\tbinom{n}{m-1}e^{n+1-m}_\xi(\delta)\circ x^{m}\\
	&=-\sum_{m=0}^{n+1}\tbinom{n+1}{m}e^{n+1-m}_\xi(\delta)\circ x^{m},
	\end{align*}
which proves the case for $k=n+1$ as desired.
The proof of Lemma~\ref{Lemma in appendix} is complete.

\end{document}